\documentclass[final,leqno]{siamltex}
\usepackage{amsmath, amssymb}
\usepackage{color}
\usepackage{graphicx}

\def\LL{L^2} 

\def\func#1{\mathop{\rm #1}\nolimits}%
\newtheorem{remark}[theorem]{Remark}

\title{Non-Conforming Multiscale Finite Element Method for Stokes Flows in Heterogeneous Media. Part I: Methodologies and Numerical Experiments}

\author{B. P. Muljadi\footnotemark[1] \footnotemark[2] \footnotemark[3]
\and J. Narski\footnotemark[2] \footnotemark[3]
\and A. Lozinski\footnotemark[4]
\and P. Degond\footnotemark[2] \footnotemark[3]
} 
\begin{document}
\maketitle
\renewcommand{\thefootnote}{\fnsymbol{footnote}}
\footnotetext[1]{Earth Science and Engineering, Imperial College London, London SW7 2AZ, UK}
\footnotetext[2]{Universit\'e de Toulouse, UPS, INSA, UT1, UTM; Institut de Math\'ematiques de Toulouse; F-31062 Toulouse, FRANCE}
\footnotetext[3]{CNRS; Institut de Math\'ematiques de Toulouse UMR 5219; F-31062 Toulouse, FRANCE}
\footnotetext[4]{Laboratoire de Math\'ematiques de Besan\c{c}on, UMR CNRS 6623, Universit\'e de Franche-Comt\'e, 25030 Besan\c{c}on Cedex, FRANCE}
\renewcommand{\thefootnote}{\arabic{footnote}}

\begin{abstract}
The Multiscale Finite Element Method (MsFEM) is developed in the vein of Crouzeix-Raviart element for solving viscous incompressible flows in genuine heterogeneous media. Such flows are relevant in many branches of engineering, often at multiple scales and at regions where analytical representations of the microscopic features of the flows are often unavailable. Full accounts to these problems heavily depend on the geometry of the system under consideration and are computationally expensive. Therefore, a method capable of solving multiscale features of the flow without confining itself to fine scale calculations is sought after. 

The approximation of boundary condition on coarse element edges when computing the multiscale basis functions critically influences the eventual accuracy of any MsFEM approaches. The weakly enforced continuity of Crouzeix - Raviart function space across element edges leads to a natural boundary condition for the multiscale basis functions which relaxes the sensitivity of our method to complex patterns of obstacles exempt from the needs of implementing any oversampling techniques. Additionally, the application of penalization method makes it possible to avoid complex unstructured domain and allows extensive use of simpler Cartesian meshes. 
\end{abstract}



\pagestyle{myheadings}
\thispagestyle{plain}
\markboth{B. P. Muljadi et al.}{}

\section{Introduction}

Stokes equations relate to many engineering practices from reservoir engineering, micro-/nano-fluidics to mechano-biological systems. Often in these fields, the problems are at multiple scales both spatially and temporally. Multiscale problems may arise due to highly oscillatory coefficients of the system or due to heterogeneity of the domain; for example complex rock matrices when modelling sub-surface flows or random placements of buildings, people and trees in the context of urban canopy flows. Full account to these systems are difficult for they depend on the geometry and often demand huge computational resources. Despite modern renaissance of high performance computing, the size of the discrete problems remains big. In some engineering contexts, it is sometimes sufficient to predict macroscopic properties of multiscale systems. Hence it is desirable to develop an efficient computational algorithm to solve multiscale problems without being confined to solving fine scale solutions. We borrow the concept of Multiscale Finite Element Method (MsFEM) \cite{thomhou} by Hou and Wu, the concept of which hinges upon the extension of multiscale basis functions pre-calculated in fine mesh to represent a ''model'' of the microscopic structure of the flow. The fact that the multiscale basis are not modelled but rather calculated extends the applicability of MsFEM to problems where the analytical representations of microscopic are unavailable.

Within the last decades, several methods sprung from similar purpose namely, Generalized finite element methods \cite{babuskaetal1}, wavelet-based homogenization method \cite{dorobantuengquist}, variational multiscale method \cite{nolenetal},various methods derived from homogenization theory \cite{bourgeat}, equation-free computations \cite{kevrekidisetal}, heterogeneous multiscale method \cite{weinanengquist} and many others. In the context of diffusion in perforated media, some studies have been done both theoretically and numerically in \cite{cioranescuetal},\cite{cioranescumurat},\cite{henningohlberger},\cite{hornung}, and \cite{lions}. For the case of advection-diffusion a method derived from heterogeneous multiscale method addressing oscillatory coefficients is studied in \cite{Deng20091549}. For viscous, incompressible flows, multiscale methods based on homogenization theory for solving slowly varying Stokes flow in porous media have been studied in \cite{brownetal, brownefendievhoang}. Several theoretical and numerical studies have been done in the couple years to address Stokes-Darcy or Stokes-Brinkman problems in vugular or fractured porous media, see \cite{arbogastetal, giraultetal, gulbransenetal, popovetal}. 

When tackling highly heterogeneous problems, it is understood that a delicate treatment is needed at the boundary condition when constructing the multiscale basis function for it greatly influences the eventual accuracy of the method under consideration. Indeed in the original work of Hou and Wu, the oversampling method was introduced to provide the best approximation of the boundary condition of the multiscale basis functions. Inaccurate approximations of these functions on the boundary would leave a method highly unreliable when dealing with arbitrary pattern of porosities or matrices.  Oversampling here means that the local problem in the course element are extended to a domain larger than the element itself, but only the interior information would be communicated to the coarse scale equation. The aim is to reduce the effect of wrong boundary conditions and bad sampling sizes. The ways in which the sampled domain is extended lead to various oversampling methods, interested readers can refer to \cite{houefendiev}, \cite{chuetal}, \cite{henningpeterseim}, \cite{efendievetal}. 

The non-conforming nature of Crouzeix-Raviart element, see \cite{CRRairo}, is shown to provide great 'flexibility' especially when arbitrary patterns of porosities are considered. In the construction of Crouzeix-Raviart multiscale basis functions, the conformity between coarse elements are not enforced in a strong sense, but rather in a weak sense i.e., the method requires merely the average of the ''jump'' of the function to vanish at coarse element edges. When very dense obstacles are introduced, which often makes it prohibitively difficult to avoid intersections between coarse element edges and obstacles, the benefit of using Crouzeix-Raviart MsFEM is significant for it allows the multiscale basis functions to have natural boundary conditions on element edges making it insensitive to complex patterns of obstacles. Moreover, the integrated application of penalization method enables one to carry the simulations onto simple Cartesian meshes. This work continues our earlier works where Crouzeix-Raviart MsFEM is implemented on diffusion and advection-diffusion problems \cite{MsFEMCR1, MsFEMCR2, muljadicrmsfem}.

\begin{figure}[hbt]
\centering
\includegraphics[width = 3.0in]{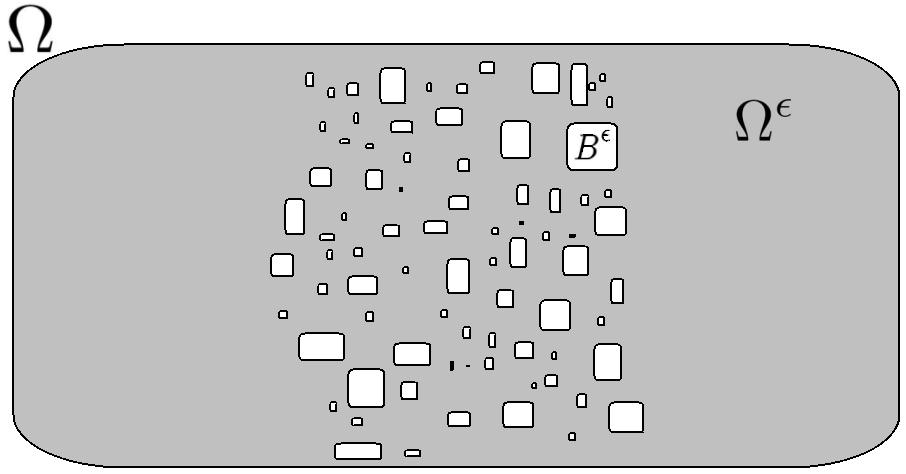} 
\caption{Illustration of domain $\Omega$ comprising fluid domain $\Omega^\epsilon$ perforated by set of obstacles $B^\epsilon$}
\label{domain}
\end{figure} 
 
This paper is organized as follows. The formulation of the problem is given in chapter \ref{sec:formulation}. The kernel of our paper is laid out in Chapter \ref{sec:crmsfem} where the Crouzeix-Raviart MsFEM is explained. In Chapter \ref{implement}, the adaptation of penalization method to our problem is discussed. Numerical tests comprising enclosed and open-channel flows with arbitrary pattern of obstacles and non-homogeneous boundary conditions are reported in Chapter \ref{sec:numerical} followed by some concluding remarks. 

\section{Formulation of Problem}
\label{sec:formulation}
We consider a Stokes problem posed in a bounded domain $\Omega \in \mathbb{R}^d$ within which a set $B^\epsilon$ of obstacles is included. The domain with voids left by obstacles is denoted $\Omega^\epsilon = \Omega \setminus B^\epsilon$ illustrated in Fig. \ref{domain}, where $\epsilon$ denotes the minimum width of obstacles. The Stokes problem is then to find $u : \Omega^\epsilon \rightarrow \mathbb{R}$ which is the solution to
\begin{eqnarray} 
 -\nu\Delta\vec{u}+\nabla p = \vec{f} & \hspace{5mm}\textrm{in}\hspace{5mm} & \Omega^\epsilon\label{main}\\
\nabla \cdot \vec{u} = 0 & \hspace{5mm}\textrm{in}\hspace{5mm} & \Omega^\epsilon \nonumber
\end{eqnarray}
The boundary value problem considered in equation \ref{main} posed on two-dimensional domain $\Omega$, together with boundary condition on $\partial\Omega$ is given by
\begin{eqnarray}
\vec{u}=0 & \hspace{5mm}\textrm{on}\hspace{5mm} & \partial B^\epsilon \cap \partial\Omega^\epsilon \label{mainBC}\\
\vec{u}=\vec{w}& \hspace{5mm}\textrm{on}\hspace{5mm} & \partial \Omega \cap \partial\Omega^\epsilon \nonumber
\end{eqnarray}
where $\vec{f} : \Omega \rightarrow \mathbb{R}$ is a given function, $\vec{w}$ is a function fixed on boundary $\partial\Omega$. In this paper, we consider only the Dirichlet boundary condition on $\partial B^\epsilon$ namely $u_{|\partial B^\epsilon} = 0$ thereby assuming that the obstacle is impenetrable. Other kinds of boundary conditions on $\partial B^\epsilon$ are subject to a completely new endeavour.

{
\subsection*{The weak formulation}

Let us restrict ourselves to the case of homogeneous
boundary conditions $\vec{w}=0$ here.\footnote{This is done only to simplify the forth-coming presentation of our MsFEM technique. Its generalization to non-homogeneous boundary data is actually straightforward and is explained in detail in Section \ref{nonhomog}.} Introduce the function spaces $%
V=(H_{0}^{1}(\Omega ^{\epsilon }))^{d}$ for the velocity, $%
M=\LL_{0}(\Omega ^{\varepsilon })=\{p\in \LL (\Omega ^{\varepsilon })$
s.t. $\int_{\Omega ^{\varepsilon }}p=0\}$ for the pressure, and $X=V\times M 
$. The weak formulation of the problem above reads: find $(\vec{u},p)\in X$
such that 
\[
c((\vec{u},p),(\vec{v},q))=\int_{\Omega ^{\varepsilon }}\vec{f}\cdot \vec{v}%
,\quad \forall (\vec{v},q)\in X 
\]%
where $c$ is the bilinear form defined by%
\begin{equation}\label{mainvar}
c((\vec{u},p),(\vec{v},q))=\int_{\Omega ^{\varepsilon }}\nabla \vec{u}%
:\nabla \vec{v}-\int_{\Omega ^{\varepsilon }}p\func{div}\vec{v}-\int_{\Omega
^{\varepsilon }}q\func{div}\vec{u} 
\end{equation}

The existence and uniqueness of the solution to problem (\ref{mainvar}) is
guaranteed by the theory of saddle point problems, especially by the inf-sup
property
\begin{equation}
\inf_{p\in M}\sup_{\vec{v}\in V}\frac{\int_{\Omega ^{\varepsilon }}p%
\mathop{\rm div}\nolimits\vec{v}}{\Vert p\Vert _{M}\Vert \vec{v}\Vert _{V}}%
\geq \gamma >0  \label{infsupb}
\end{equation}%
which is known to hold on any domain with Lipschitz boundary \cite{girault}.
A nice reformulation of this theory is presented in \cite{ern}, where it is
proved that (\ref{infsupb}) implies also the inf-sup property for the form $%
c $ 
\begin{equation}
\inf_{(\vec{u},p)\in X}\sup_{(\vec{v},q)\in X}\frac{c((\vec{u},p),(\vec{v}%
,q))}{\Vert \vec{u},p\Vert _{M}\Vert \vec{v},q\Vert _{X}}\geq \gamma _{c}>0
\label{infsupc}
\end{equation}%
with a constant $\gamma _{c}$ that depends only on the constant $\gamma $ in
(\ref{infsupb}). One invokes than the BNB theorem that states that the
variational problem (\ref{mainvar}) has the unique solution provided the
symmetric bilinear form $c$ is bounded (which is evident in our case) and
satisfies the inf-sup property (\ref{infsupc}). Note that the second
hypothesis in the BNB theorem of \cite{ern} is void in the case of problems
with a symmetric bilinear form.

\begin{remark}
In what follows, we shall sometimes extend the functions defined on $\Omega^\epsilon$ and vanishing on $\partial B^\epsilon$ to the whole domain $\Omega$ by setting them to 0 on $B^\epsilon$. From now on, we shall identify such functions with their extended versions without further notice. For example, an alternative definition of the space $V$ can be written as 
$$
V=\{\vec{v}\in H_{0}^{1}(\Omega)^{d} \text{ such that }\vec{v}=0 \text{ on }B^\epsilon \} 
$$
\end{remark}
}

\begin{figure}[hbt]
\centering
\includegraphics[width = 2.5in]{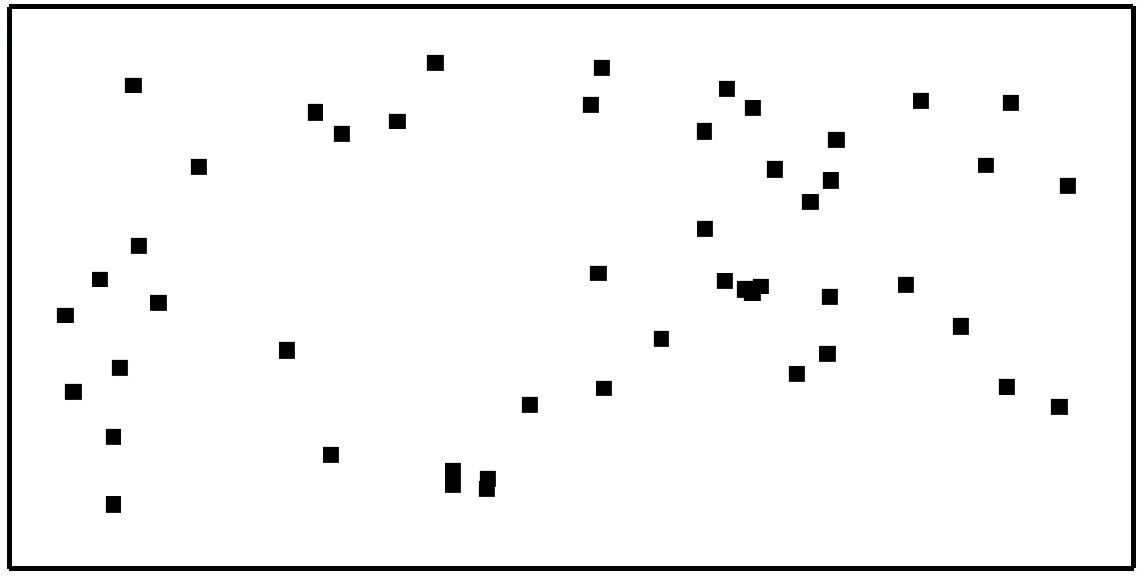} 
\caption{Computational domain with $49$ arbitrarily placed obstacles for cavity flow}
\label{cavityperfor}
\end{figure} 
\begin{figure}[hbt]
\centering
\includegraphics[width = 5in]{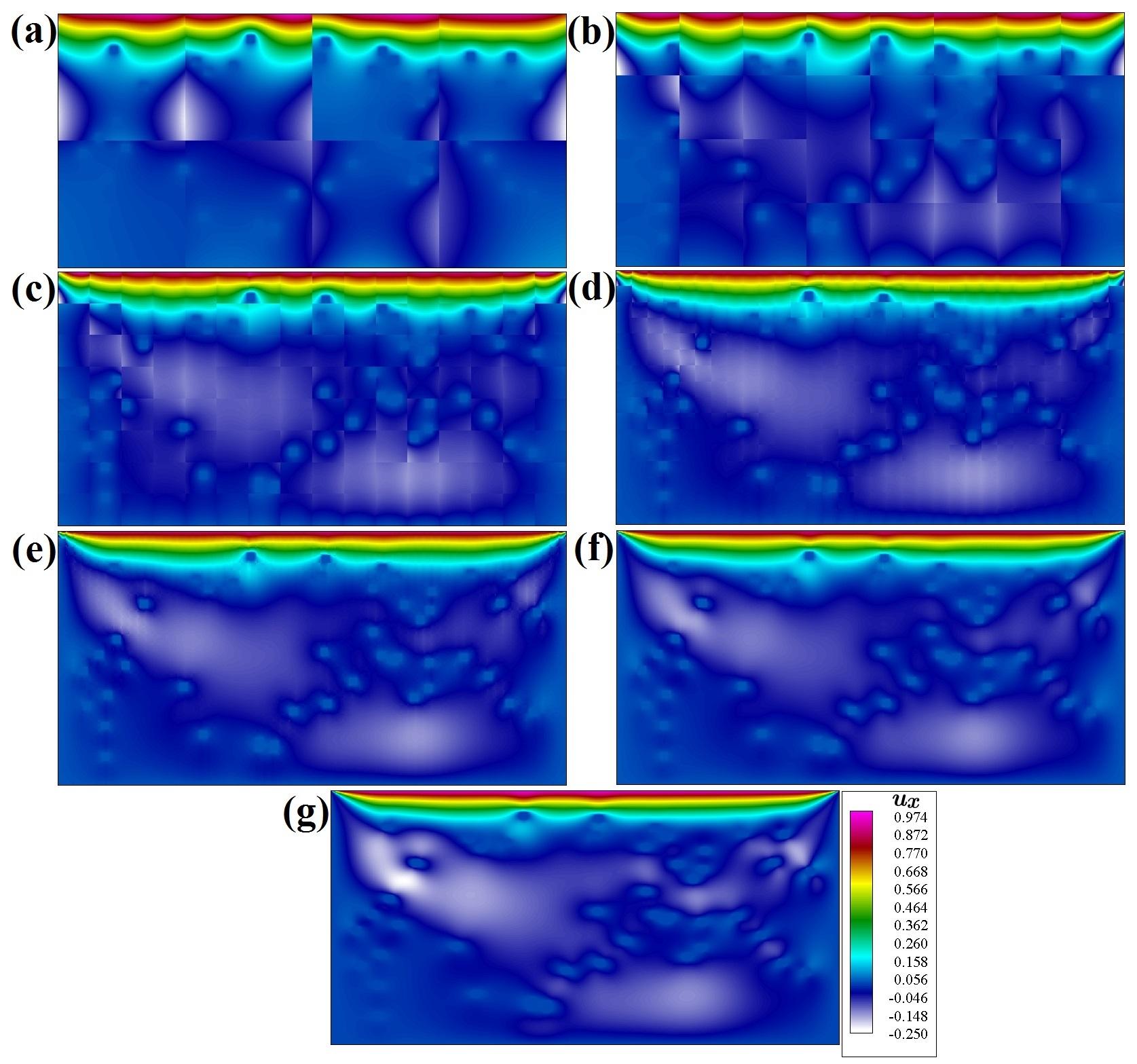} 
\caption{$u_x$ contour of a cavity flow (a) $2\times 4$, (b) $4\times 8$, (c) $8\times 16$, (d)$16\times32$, (e)$32\times 64$, (f) $64\times 128$, (g) Reference solution calculated on $640\times 1280$, with 49 arbitrarily placed obstacles}
\label{cavityux}
\end{figure} 
\begin{figure}[hbt]
\centering
\includegraphics[width = 5in]{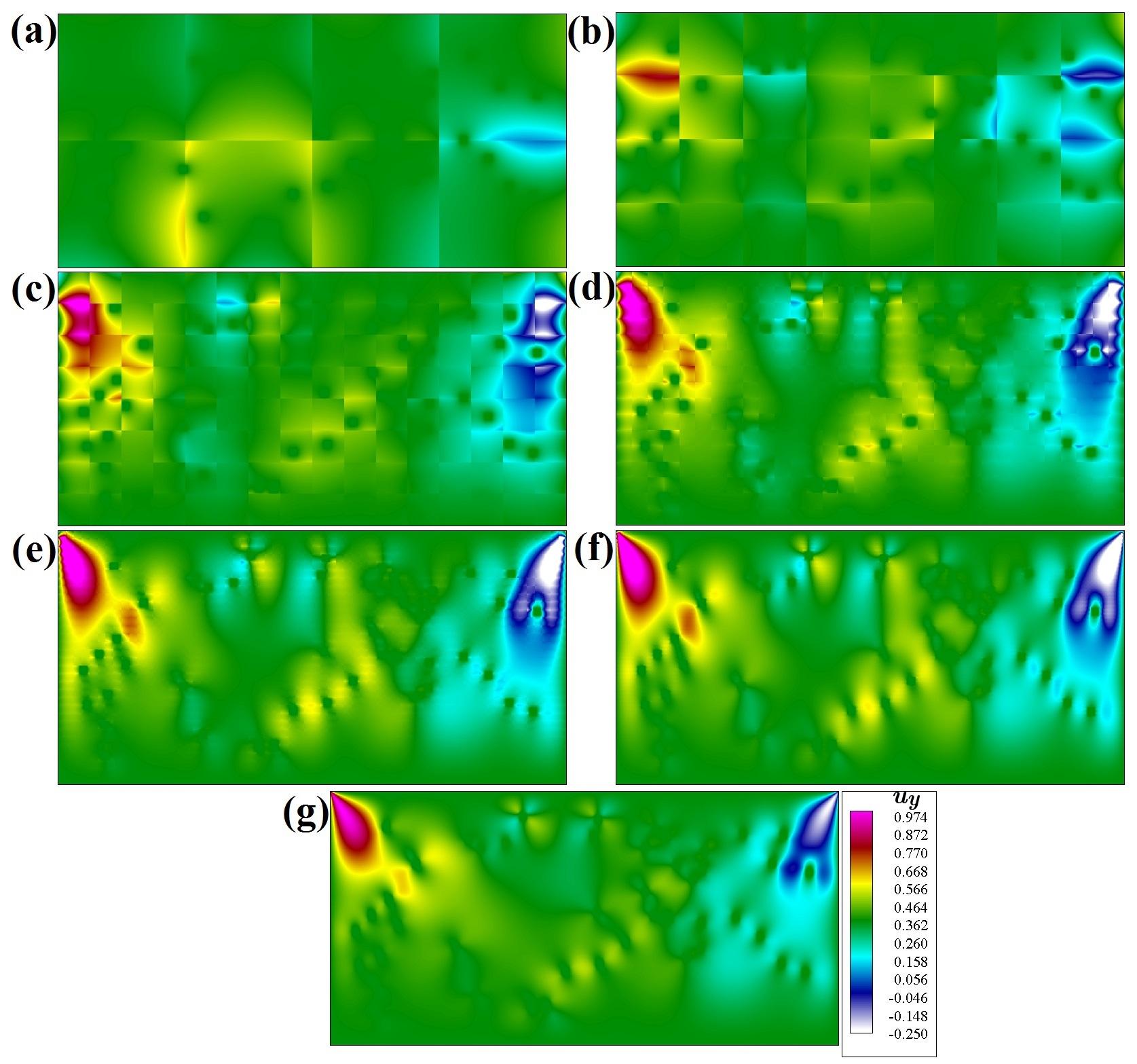} 
\caption{$u_y$ contour of a cavity flow (a) $2\times 4$, (b) $4\times 8$, (c) $8\times 16$, (d)$16\times32$, (e)$32\times 64$, (f) $64\times 128$, (g) Reference solution calculated on $640\times 1280$, with 49 arbitrarily placed obstacles}
\label{cavityuy}
\end{figure} 
\begin{figure}[hbt]
\centering
\includegraphics[width = 3in]{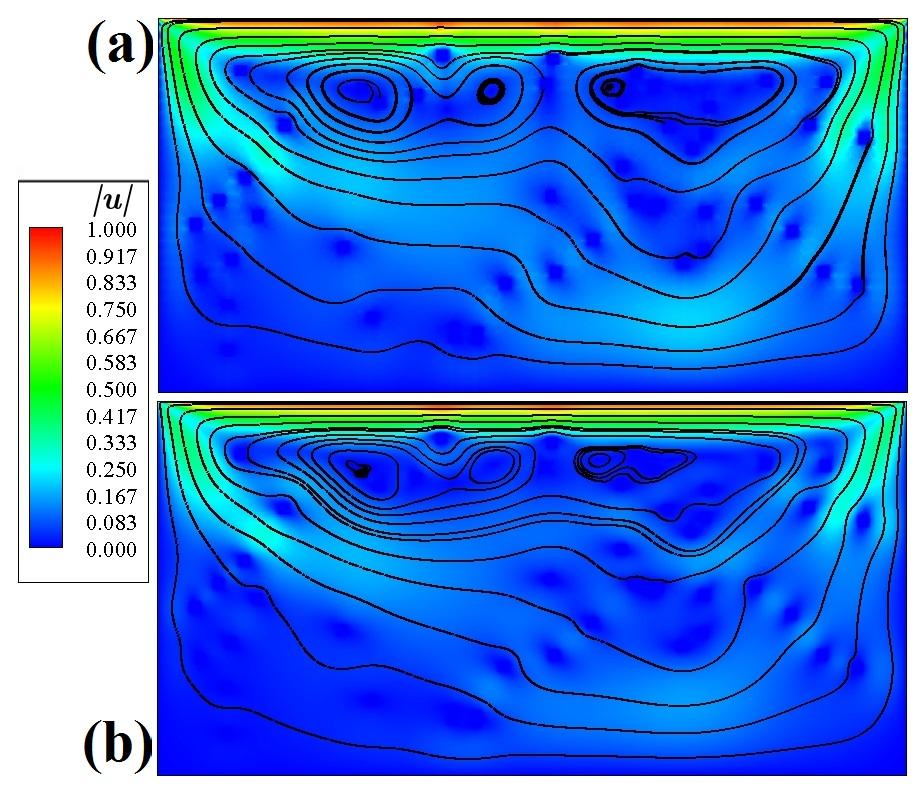} 
\caption{$|u|$ contour of a cavity flow with streamlines calculated with Crouzeix-Raviart MsFEM on (a) $32\times 64$ elements and (b) reference solution}
\label{cavitystream}
\end{figure} 
\begin{figure}[hbt]
\centering
\includegraphics[width = 5in]{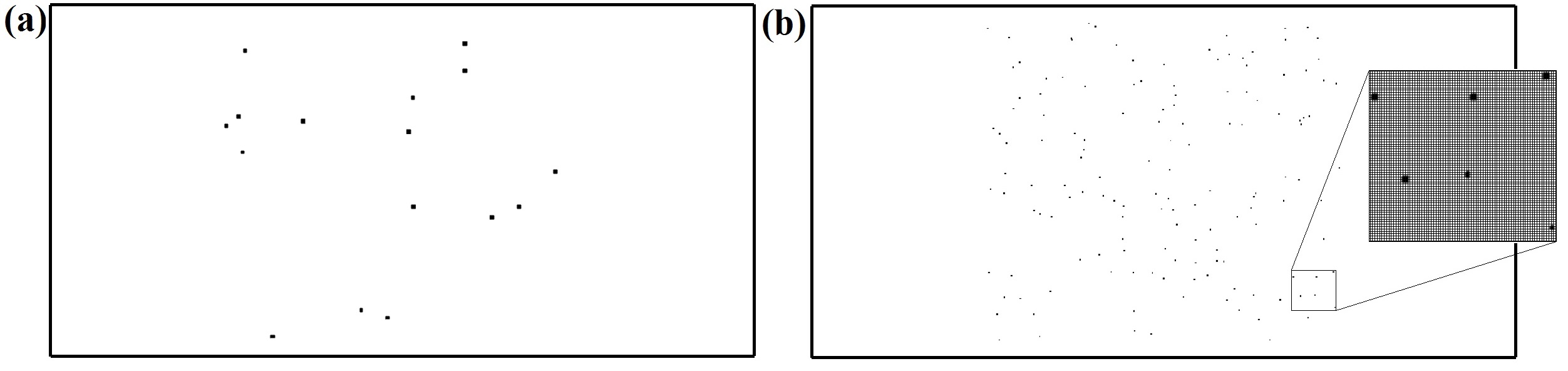} 
\caption{Computational domain with arbitrarily placed obstacles for open-channel flow (a) case A. With $16$ obstacles (b) case B. With $144$ obstacles (the enlarged figure is gridded to illustrate the size of fine obstacles with respect to fine elements)}
\label{channelperfor}
\end{figure} 
\begin{figure}[hbt]
\centering
\includegraphics[width = 5in]{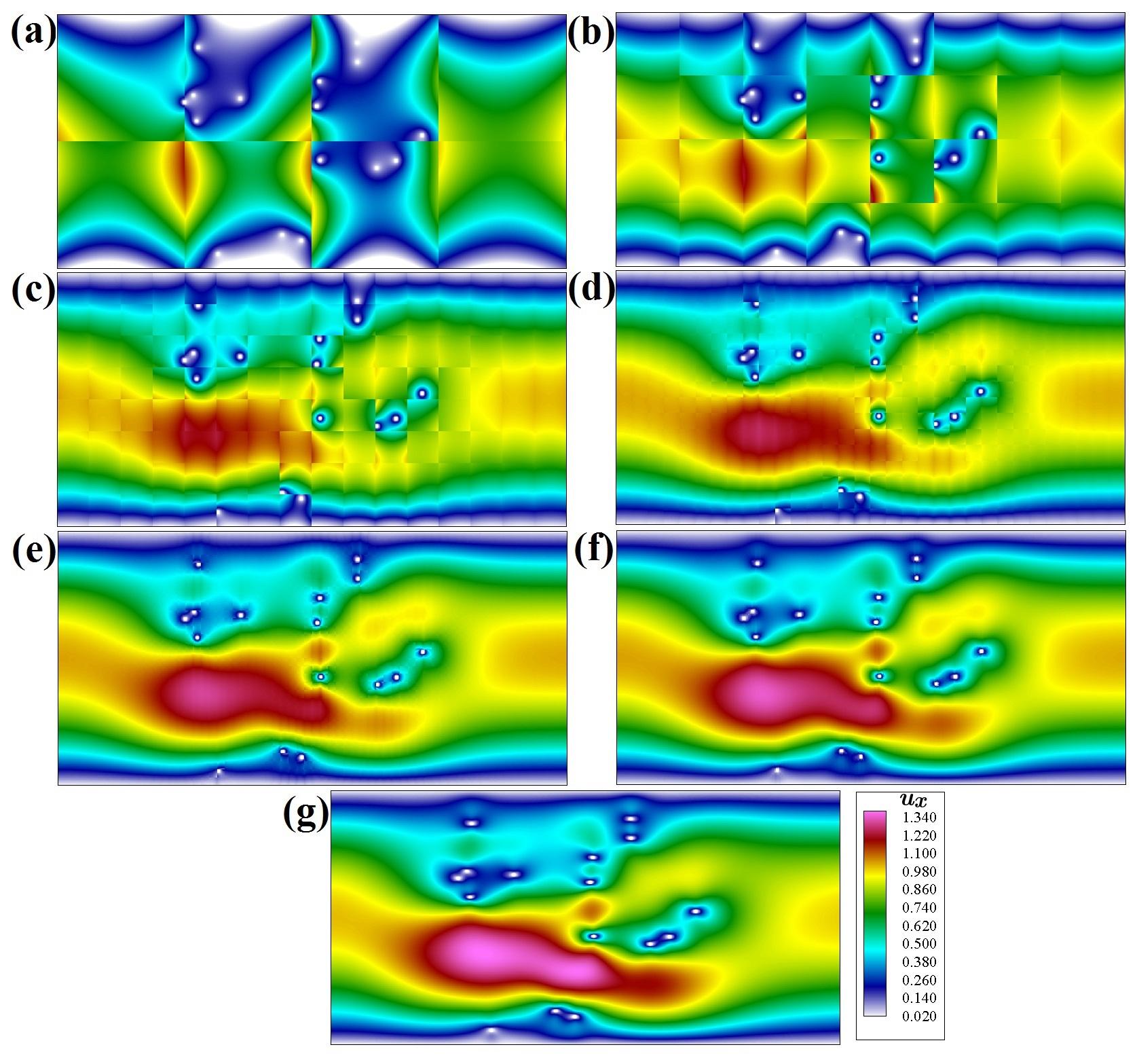} 
\caption{$u_x$ contour of a channel flow (a) $2\times 4$, (b) $4\times 8$, (c) $8\times 16$, (d)$16\times32$, (e)$32\times 64$, (f) $64\times 128$, (g) Reference solution calculated on $640\times 1280$, with 16 arbitrarily placed obstacles}
\label{channeluxa}
\end{figure} 
\begin{figure}[hbt]
\centering
\includegraphics[width = 5in]{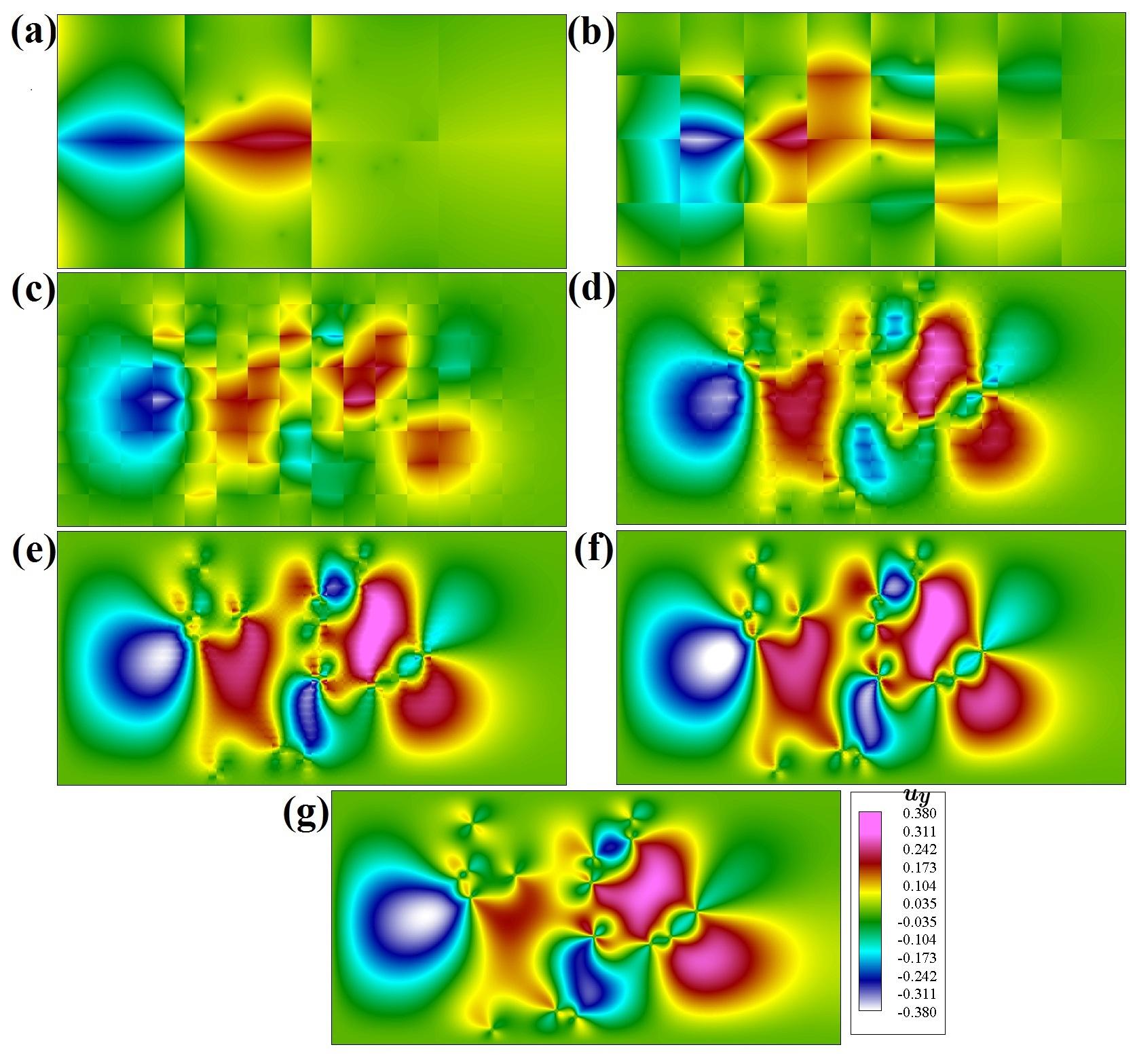} 
\caption{$u_y$ contour of a channel flow (a) $2\times 4$, (b) $4\times 8$, (c) $8\times 16$, (d)$16\times32$, (e)$32\times 64$, (f) $64\times 128$, (g) Reference solution calculated on $640\times 1280$, with 16 arbitrarily placed obstacles}
\label{channeluya}
\end{figure}
\begin{figure}[hbt]
\centering
\includegraphics[width = 3in]{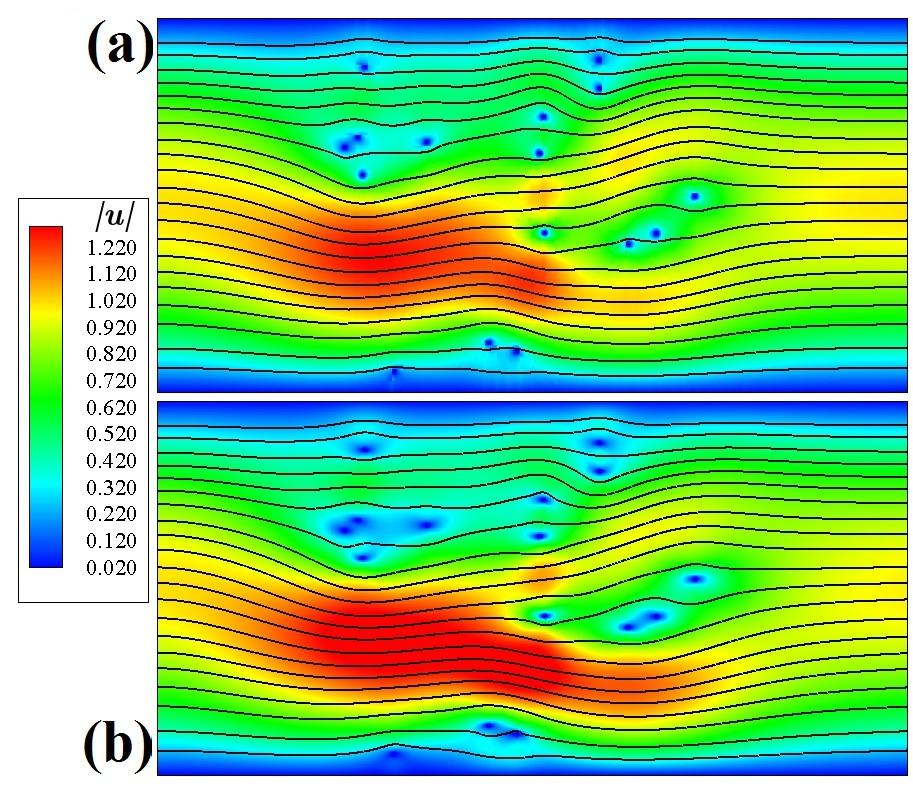} 
\caption{$|u|$ contour of a channel flow with streamlines calculated with Crouzeix-Raviart MsFEM on (a) $32\times 64$ elements and (b) reference solution with 16 arbitrarily placed obstacles}
\label{channelstreama}
\end{figure} 
\begin{figure}[hbt]
\centering
\includegraphics[width = 5in]{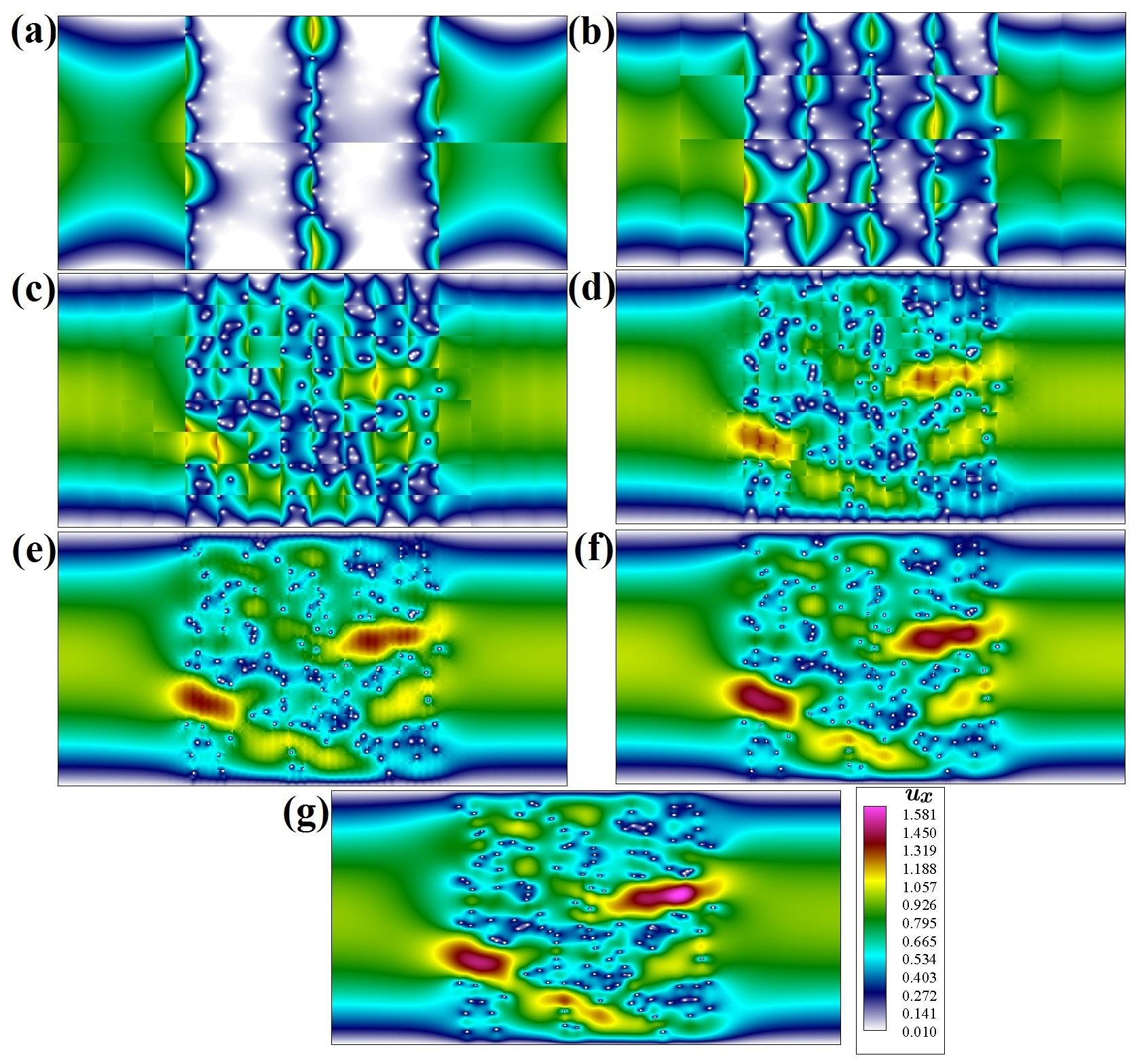} 
\caption{$u_x$ contour of a channel flow (a) $2\times 4$, (b) $4\times 8$, (c) $8\times 16$, (d)$16\times32$, (e)$32\times 64$, (f) $64\times 128$, (g) Reference solution calculated on $640\times 1280$, with 144 arbitrarily placed obstacles}
\label{channelux}
\end{figure} 
\begin{figure}[hbt]
\centering
\includegraphics[width = 5in]{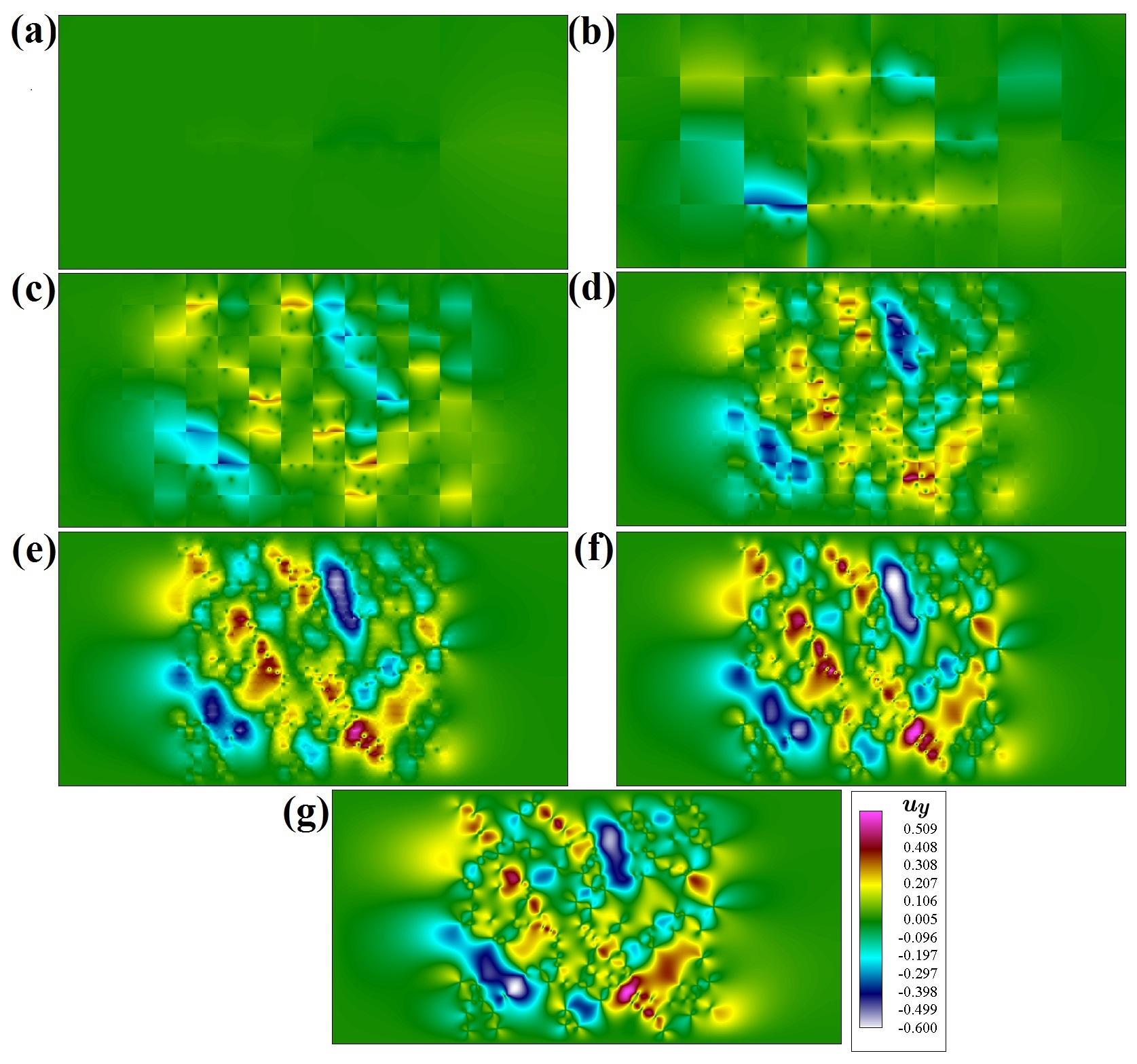} 
\caption{$u_y$ contour of a channel flow (a) $2\times 4$, (b) $4\times 8$, (c) $8\times 16$, (d)$16\times32$, (e)$32\times 64$, (f) $64\times 128$, (g) Reference solution calculated on $640\times 1280$, with 144 arbitrarily placed obstacles}
\label{channeluy}
\end{figure}
\begin{figure}[hbt]
\centering
\includegraphics[width = 3in]{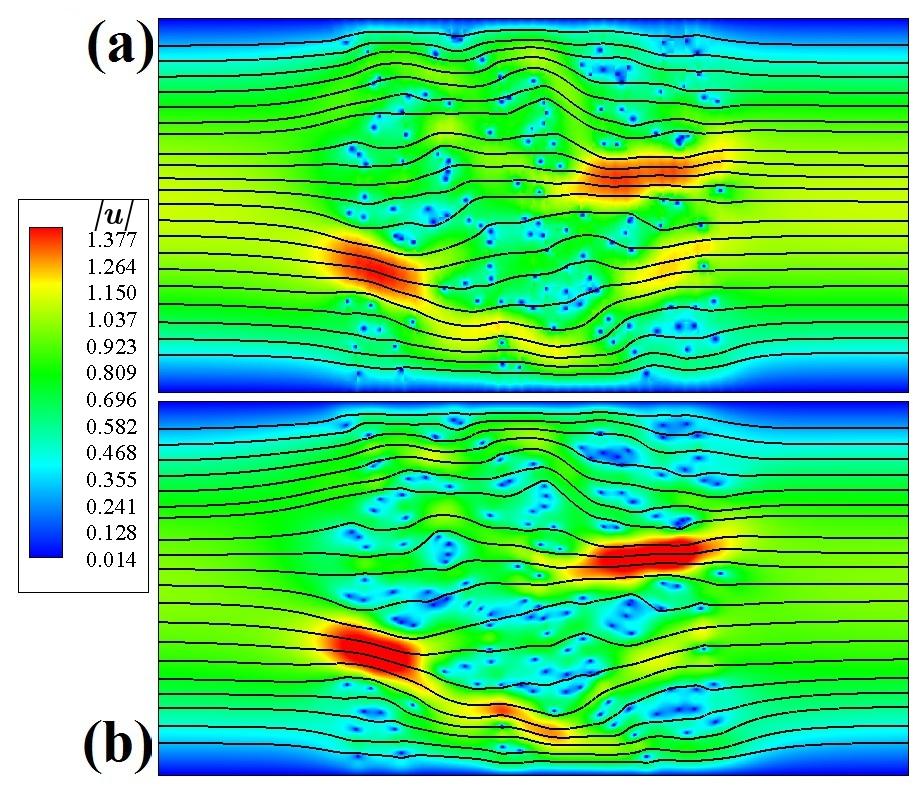} 
\caption{$|u|$ contour of a channel flow with streamlines calculated with Crouzeix-Raviart MsFEM on (a) $32\times 64$ elements and (b) reference solution with 144 arbitrarily placed obstacles}
\label{channelstream}
\end{figure} 

\section{Crouzeix-Raviart MsFEM space for Stokes equations}
\label{sec:crmsfem}
Let us introduce a mesh $\mathcal{T}_H$ on $\Omega$ consisting of $N_H$ polygons/polyhedrons of diameter at most $H$. 
Let $\mathcal{E}_H$ denote the set of all the edges/faces of $\mathcal{T}_H$ including those on the domain boundary $\partial\Omega$. It is assumed that the mesh does not contain any hanging nodes and each edge is shared by two elements except those on $\partial\Omega$ which belong only to one element. We also assume that the mesh $\mathcal{T}_H$ is regular, i.e. for any mesh element $T \in \mathcal{T}_H$, there exists a smooth one-to-one mapping $\mathcal{M} : \tilde{T} \rightarrow T$ where $\tilde{T} \subset \mathbb{R}^d$ is the element of reference, and that $\parallel \nabla\mathcal{M} \parallel_{L^\infty} \leq DH$, $\parallel \nabla\mathcal{M}^{-1} \parallel_{L^\infty} \leq DH^{-1}$ with $D$ being universal constant independent of $T$. From now on, the elements of  $\mathcal{T}_H$ will be referred to as the triangles and the elements of  $\mathcal{E}_H$ will be referred to as the edges, although all the reasoning makes perfect sense in a general situation of any ambient dimension and a mesh consisting of polygons/polyhedrons of any shape.    

\subsection{Definition of the MsFEM space}{
To construct the MsFEM space, we proceed as in our previous work \cite{MsFEMCR1, MsFEMCR2} by introducing first the extended velocity space 
\[
V_{H}^{ext}=\left\{ 
\begin{array}{c}
\vec{u}\in (\LL (\Omega ))^{d}\text{ such that }\vec{u}|_{T}\in
(H^{1}(T))^{d}\text{ for any }T\in \mathcal{T}_{H},
\  \vec{u}=0\text{ on }B^\epsilon,\\ 
\noalign{\vskip 3pt}\displaystyle\int_{E}[[u]]=0\text{ for all }E\in 
\mathcal{E}_{H}%
\end{array}%
\right\} , 
\]%
where $[[u]]$ denotes the jump of~$u$ across an internal edge and $[[u]]=u$
on the boundary $\partial \Omega $. The idea behind this space is to enhance
the natural velocity space $(H_{0}^{1}(\Omega ))^{d}$ so that to have at our
disposal the vector fields discontinuous across the edges of the
mesh since our goal is to construct a non conforming approximation method. The continuity is preserved only in the weak sense by requiring that
the average is conserved across any edge. 

We want now to decompose the
extended velocity-pressure space $X_{H}^{ext}=V_{H}^{ext}\times M$ into the
direct sum of a finite dimensional subspace $X_{H}$ of coarse scales,
which will be used for approximation, and an infinitely dimensional subspace 
$X_{H}^{0}$ of unresolved fine scales 
\begin{equation}
X_{H}^{ext}=X_{H}\oplus X_{H}^{0}.  \label{XHdirect}
\end{equation}%
More specifically, we introduce the fine scale subspace as $%
X_{H}^{0}=V_{H}^{0}\times M_{H}^{0}$ with%
\begin{eqnarray*}
V_{H}^{0} &=&\left\{ \displaystyle u\in V_{H}^{ext}\text{ such that }%
\int_{E}u=0\text{ for all }E\in \mathcal{E}_{H}\right\} \\
M_{H}^{0} &=&\left\{ \displaystyle p\in M\text{ such that }\int_{T}p=0\text{
for all }T\in \mathcal{T}_{H}\right\} \text{ }
\end{eqnarray*}%
The subspace $X_{H}$ is then defined as the "orthogonal" complement of $%
X_{H}^{0}$ with respect to the bilinear form $c:$%
\begin{equation}
(\vec{u}_{H},p_{H})\in X_{H}~\iff c((\vec{u}_{H},p_{H}),(\vec{v},q))=0,\quad
\forall (\vec{v},q)\in X_{H}^{0}  \label{XHdef}
\end{equation}%
Two remarks are in order to clarify this definition:

\begin{itemize}
\item 
The bilinear form $c$ is applied in the formula above to the functions from $V_H^{ext}$ which are discontinuous across the edges of the mesh, and thus non differentiable. To bypass this difficulty, the integrals in the definition of $c$ should be understood as $\int_{\Omega
^{\varepsilon }}\cdots =\sum_{T\in \mathcal{T}_{H}}\int_{\Omega
^{\varepsilon }\cap T}\cdots $. The same convention will be implicitly employed from now on if needed, as will be clear form the context.   
\item 
We have put the word "orthogonal" in quotes since the bilinear form $c$ is
not a scalar product (not positive definite). We shall prove however in the following lemma that
the subspace $X_{H}\subset X_{H}^{ext}$ defined by (\ref{XHdef}) forms indeed a direct sum with $X_H^0$.
\end{itemize}

\begin{lemma}
Let the functional spaces $M_{H}\subset M$ and $V_{H}\subset V_{H}^{ext}$ be
defined as
\begin{equation}\label{MHdef}
M_{H}=\{{q}\in \LL_{0}(\Omega )\,\text{\ such that }\,q|_{T}=const,~%
\forall T\in \mathcal{T}_{H}\}
\end{equation}
\begin{eqnarray}\label{VHdef}
V_{H} &=&\{\vec{v}:(\LL (\Omega ))^{d}\,\text{ }:\,\forall T\in \mathcal{T}%
_{H}~\exists s\in L_{0}^{d}(\,\Omega ^{\varepsilon }\cap T)~\text{such that}
\\
&&-\Delta \vec{u}_{H}+\nabla s=0\text{ on }\Omega ^{\varepsilon }\cap T, 
\nonumber \\
&&\func{div}\vec{u}_{H} =const\text{ on }\Omega ^{\varepsilon }\cap T 
\nonumber \\
&&\vec{u}_{H}=0\text{ on }B^{\varepsilon }\cap T  \nonumber \\
&&n\cdot \nabla \vec{u}_{H}-sn=const\text{ on }E\cap \Omega ^{\varepsilon },
\  \forall E\in\mathcal{E}(T)\}  \nonumber
\end{eqnarray}%
where $\mathcal{E}(T)$ is the ensemble of edges composing $\partial T$.
Let, for any $\vec{u}_{H}\in V_{H}$, glue together the functions $s$ on triangles $T\in 
\mathcal{T}_{H}$ in the definition above into a single function $\pi _{H}(\vec{u}%
_{H})\in M_H^0$ such that $\pi _{H}(\vec{u}_{H})=s$ on any triangle $T\in \mathcal{T}_{H}$.
Then, $\pi _{H}:V_{H}\rightarrow M_{H}^{0}$ is a well defined linear operator. The space 
$X_H$ defined by (\ref{XHdef}) can be represented as
\begin{equation}
X_{H}=span\{(\vec{u}_{H},\pi _{H}(\vec{u}_{H})),~\vec{u}_{H}\in
V_{H}\}\oplus span\{(0,\bar{p}_{H}),~\bar{p}_{H}\in M_{H}\}  \label{XHdef1}
\end{equation}%
Moreover, the relation (\ref{XHdirect}) holds true.
\end{lemma}

\begin{proof}
Let $(\vec{u}_{H},p_{H})\in X_{H}$ in the sense of definition (\ref{XHdef}).
We can decompose $p_{H}$ as%
\[
p_{H}=\bar{p}_{H}+p_{H}^{\prime }\text{ with }\bar{p}_{H}\in M_{H}\text{ and 
}p_{H}^{\prime }\in M_{H}^{0}\text{.} 
\]%
This decomposition is unique since the value $\bar{p}_{H}$ on any triangle $%
T\in \mathcal{T}_{H}$ is simply the average of $p_{H}$ on this triangle.
Noting that 
\[
\int_{\Omega ^{\varepsilon }}\bar{p}_{H}\func{div}\vec{v}
=\sum_{T\in\mathcal{T}_{H}}\bar{p}_{H}|_{T}\int_{T}\func{div}\vec{v}
=\sum_{T\in \mathcal{T}_{H}}\bar{p}_{H}|_{T}\int_{\partial T}\bar{v}%
\cdot n=0 
\]%
for any $\bar{v}\in V_{H}^{0}$, we can rewrite the definition (\ref{XHdef})
as%
\begin{equation}
c((\vec{u}_{H},p_{H}),(\vec{v},q))=\int_{\Omega ^{\varepsilon }}\nabla \vec{u%
}_{H}:\nabla \vec{v}-\int_{\Omega ^{\varepsilon }}p_{H}^{\prime }\func{div}%
\vec{v}-\int_{\Omega ^{\varepsilon }}q\func{div}\vec{u}_{H}=0  \label{corto}
\end{equation}%
for any $\vec{v}\in V_{H}^{0}$ and $q\in M_{H}^{0}$. 

Choosing a triangle $T\in \mathcal{T}_{H}$ and considering the test functions $\vec{v}=0$, $q\in
\LL_{0}(T\cap \Omega ^{\varepsilon })$ with $q$ vanishing outside $T$
yields 
\[
\int_{\Omega ^{\varepsilon } \cap T}q\func{div}\vec{u}_{H}=0\quad \forall q\in
\LL_{0}(T\cap \Omega ^{\varepsilon }),\text{ i.e. }\func{div}\vec{u}%
_{H}=const\text{ on }T\cap \Omega ^{\varepsilon }. 
\]
We now observe that for any edge $E\in \mathcal{E}(T)$ there are functions $%
\vec{v}_{E,i}\in H^1(T)$, $i=1,\ldots,d$, such that $\vec{v}_{E,i}$ vanishes on $B^\epsilon\cap T$ and $\int_{E}\vec{v}%
_{E,i}=e_{i}$, with $\{e_{1},\ldots,e_{d}\}$ being the canonical basis of $\mathbb{R}^{d}$, and $\int_{E^{\prime }}\vec{v}_{E,i}=0$ for all $E^{\prime }\in 
\mathcal{E}(T)$, $E^{\prime }\not=E$. 
The space of functions in $H^{1}(T)$ vanishing on $%
B^{\epsilon }\cap T$ can be represented as%
\begin{eqnarray*}
V(T) &:&=\{\vec{v}:(H^{1}(T))^{d}\ \text{\ such that }\vec{v}=0\text{ on }%
B^{\varepsilon }\cap T\} \\
&=&V_{\int 0}(T)\oplus span\{\vec{v}_{E,i}\text{, }E\in \mathcal{E}_{H}\text{%
, }i=1,\ldots,d\}.
\end{eqnarray*}%
where 
\begin{equation*}
V_{\int 0}(T)=\{\vec{v}:(H^{1}(T))^{d}:\int_{E}\vec{v}=0,~\forall E\in 
\mathcal{E}(T)\text{ and }\vec{v}=0\text{ on }B^{\varepsilon }\cap T\}.
\end{equation*}%
Denoting for any $E\in \mathcal{E}(T)$ and $i=1,\ldots,d$%
\begin{equation*}
\lambda _{E,i}=\int_{\Omega ^{\varepsilon }}\nabla \vec{u}_{H}:\nabla \vec{v}%
_{E,i}-\int_{\Omega ^{\varepsilon }}p_{H}^{\prime }\mathop{\rm div}\nolimits%
\vec{v}_{E,i}
\end{equation*}%
and taking into account relation (\ref{corto}) with any $\vec{v}\in V_{\int
0}(T)$ extended by 0 outside $T$ and $q=0$ we see that%
\begin{equation*}
\int_{\Omega ^{\varepsilon }\cap T}\nabla \vec{u}_{H}:\nabla \vec{v}%
-\int_{\Omega ^{\varepsilon }\cap T}p_{H}^{\prime }\mathop{\rm div}\nolimits%
\vec{v}=\sum_{E\in \mathcal{E}_{H}}\vec{\lambda}_{E}\cdot \int_{E}\vec{v}%
,\quad \forall \vec{v}\in V(T)
\end{equation*}%
with $\vec{\lambda}_{E}=(\lambda _{E,1},\ldots,\lambda _{E,d})^{T}$. Integrating by
parts converts this into the strong form: for any $T\in \mathcal{T}_{H}$%
\begin{eqnarray*}
-\Delta \vec{u}_{H}+\nabla p_{H}^{\prime } &=&0\text{ on }\Omega
^{\varepsilon }\cap T \\
\vec{u}_{H} &=&0\text{ on }B^{\varepsilon }\cap T \\
n\cdot \nabla \vec{u}_{H}-p_{H}^{\prime }n &=&const\text{ on }E\cap \Omega
^{\varepsilon }\text{ for all edges }E\text{ of }\partial T
\end{eqnarray*}%
We see thus that $\vec{u}_{H}\in V_{H}$ and $p_{H}^{\prime }$ is uniquely
determined by $\vec{u}_{H}$ (indeed, for any $\vec{u}_{H}$ fixed in the
formulas above, the gradient $\nabla p_{H}^{\prime }$ is uniquely determined
on any triangle by the equation in the first line, and the overage $%
p_{H}^{\prime }$ over any triangle is 0). We have thus $p_{H}^{\prime }=\pi
_{H}(\vec{u}_{H})~$with $\vec{u}_{H}\in V_{H}$, which proves that $(\vec{u}%
_{H},p_{H})$ belongs to the space defined by (\ref{XHdef1}).

Reversing the arguments above, we prove easily that any $(\vec{u}%
_{H},p_{H})\in X_{H}$ in the sense of definition (\ref{XHdef1}) satisfies
also relation (\ref{XHdef}). We conclude thus that the definitions (\ref%
{XHdef}) an (\ref{XHdef1}) are equivalent.

It remains to prove $X_{H}^{ext}=X_{H}+X_{H}^{0}$, i.e. that
any $(\vec{u},p)\in X_{H}^{ext}$ can be represented as 
\[
\vec{u}=\vec{u}_{H}+\vec{u}^{0},\quad p_{H}=\pi _{H}(\vec{u}_{H})+\bar{p}%
_{H}+p^{0}
\]%
with some $\vec{u}_{H}\in V_{H}$, $\vec{u}^{0}\in V_{H}^{0}$, $\bar{p}%
_{H}\in M_{H}$ and $p^{0}\in M_{H}^{0}$. This is equivalent to the following
statement: for any $(\vec{u},p)\in X_{H}^{ext}$ there exists $(\vec{u}%
^{0},p^{0})\in X_{H}^{0}$ such that%
\begin{equation}
c((\vec{u}^{0},p^{0}),(\vec{v},q))=c((\vec{u},p),(\vec{v},q)),\quad \forall (%
\vec{v},q)\in X_{H}^{0}.  \label{cu0p0}
\end{equation}%
In order to prove the existence of such $(\vec{u}^{0},p^{0})$, we pick up
any triangle $T\in \mathcal{T}_{H}$ and remark that
the restriction of $(\vec{u}^{0},p^{0})$ to the triangle $T$ belongs to
$V_{\int 0}(T)\times \LL_{0}(T\cap \Omega ^{\varepsilon })$ and
satisfies%
\begin{eqnarray*}
\int_{T\cap \Omega ^{\varepsilon }}\nabla \vec{u}^{0} :\nabla \vec{v}%
-\int_{T\cap \Omega ^{\varepsilon }}p^{0}\func{div}\vec{v} &=& \int_{T\cap
\Omega ^{\varepsilon }}\nabla \vec{u}:\nabla \vec{v}-\int_{T\cap \Omega
^{\varepsilon }}p\func{div}\vec{v},\quad \forall \vec{v}\in V_{\int 0}(T) \\
\int_{T\cap \Omega ^{\varepsilon }}q\func{div}\vec{u}^{0} &=&\int_{T\cap
\Omega ^{\varepsilon }}q\func{div}\vec{u},\quad \forall q\in \LL_{0}(T)
\end{eqnarray*}%
This is a standard saddle point problem and its solution exists since we
obviously have the inf-sup property%
\[
\inf_{q\in \LL_{0}(T\cap \Omega ^{\varepsilon })}\sup_{\vec{v}\in V_{\int
0}(T)}\frac{\int_{T\cap \Omega ^{\varepsilon }}q\func{div}\vec{v}}{\Vert
q\Vert _{\LL (T)}|\vec{v}|_{H^{1}(T)}}>0.
\]

Finally, it is easy to see that $X_{H}\cap X_{H}^{0}=\{0\}$, i.e. the
relation (\ref{XHdirect}) holds true. Indeed, if $(\vec{u},p)\in X_{H}\cap
X_{H}^{0}$ then%
\begin{equation}
c((\vec{u},p),(\vec{v},q))=0  \label{cup00}
\end{equation}%
both for $(\vec{v},q)\in X_{H}$ and for $(\vec{v},q)\in X_{H}^{0}$. Since $%
X_{H}^{ext}=X_{H}+X_{H}^{0}$, we have (\ref{cup00}) for any $(\vec{v},q)\in
X_{H}^{ext}$which implies $(\vec{u},p)=0$ by the inf-sup property (\ref{infsupc}).
\end{proof}

\subsection{Basis functions for the space $V_H$}

The following lemma shows that one can construct a basis for $V_{H}$
consisting of functions associated to the edges of the mesh. Each basis
function is supported in the patch $\omega _{E}$ consisting of 2 triangles
ajacent to an edge $E\in \mathcal{E}_{H}$ as in the classical
Crouzeix-Raviart FEM$.$

\begin{lemma}
For any edge $E\in \mathcal{E}_{H}$ one can construct $\vec{\Phi}_{E,i}\in
V_{H}$, $i=1,\ldots,d$, such that $\int_{E}\vec{\Phi}_{E,i}=\vec{e}_{i}$, with $\{%
\vec{e}_{1},\ldots,\vec{e}_{d}\}$ being the canonical basis of $\mathbb{R}^{d}$,
and $\int_{E^{\prime }}\vec{\Phi}_{E,i}=0$ for all $E^{\prime }\in \mathcal{E%
}_{H}$, $E^{\prime }\not=E$. These functions form a basis of $V_{H}:$%
\begin{equation*}
V_{H}=span\{\vec{\Phi}_{E,i},~E\in \mathcal{E}_{H},~i=1,\ldots,d\}.
\end{equation*}%
Moreover, $\sup p$ $(\vec{\Phi}_{E,i})\subset \omega _{E}$, i.e. the
ensemble of 2 triangles from $\mathcal{T}_{H}$ ajacent to the edge $E\in 
\mathcal{E}_{H}$.
\end{lemma}

\begin{proof}
For any edge $E\in \mathcal{E}_{H}$ there exist functions $\vec{v}_{E,i}\in
V_{H}^{ext}$, $i=1,\ldots,d$, such that $\int_{E}\vec{v}_{E,i}=e_{i}$ and $%
\int_{E^{\prime }}\vec{v}_{E,i}=0$ for all $E^{\prime }\in \mathcal{E}_{H}$, 
$E^{\prime }\not=E.$The space $V_{H}^{ext}$ can be evidently decomposed as%
\begin{equation*}
V_{H}^{ext}=V_{H}^{0}\oplus span\{\vec{v}_{E,i}\text{, }E\in \mathcal{E}_{H}%
\text{, }i=1,\ldots,d\}.
\end{equation*}%
We also have the decomposition 
\begin{equation}
V_{H}^{ext}=V_{H}^{0}\oplus V_{H}  \label{dirv1}
\end{equation}%
which implies for any $E\in \mathcal{E}_{H}$ \ and $i=1,\ldots,d$ that there exist
functions $\vec{\Phi}_{E,i}\in V_{H}$ and $\vec{v}_{E,i}^{0}\in V_{H}^{ext}$
such that%
\begin{equation*}
\vec{v}_{E,i}=\vec{v}_{E,i}^{0}+\vec{\Phi}_{E,i}.
\end{equation*}%
Thus%
\begin{equation}
V_{H}^{ext}=V_{H}^{0}\oplus span\{\vec{\Phi}_{E,i}\text{, }E\in \mathcal{E}%
_{H}\text{, }i=1,\ldots,d\}.  \label{dirv2}
\end{equation}%
Comparing (\ref{dirv1}) with (\ref{dirv2}) implies%
\begin{equation*}
V_{H}=span\{\vec{\Phi}_{E,i}\text{, }E\in \mathcal{E}_{H}\text{, }i=1,\ldots,d\}.
\end{equation*}

It remains to prove that the support of $\vec{\Phi}_{E,i}$ is indeed within
the patch $\omega _{E}$. To this end, consider the function  $\vec{\Phi}%
_{E,i}^{\prime }\in V_{H}^{ext}$ such that $\vec{\Phi}_{E,i}^{\prime }=\vec{%
\Phi}_{E,i}$ on  $\omega _{E}$ and $\vec{\Phi}_{E,i}^{\prime }=0$ outside $%
\omega _{E}$. According to definition (\ref{VHdef}), $\vec{\Phi}%
_{E,i}^{\prime }\in V_{H}$. Moreover, $\vec{\Phi}_{E,i}^{\prime }-\vec{\Phi}%
_{E,i}\in V_{H}^{0}$ so that $\vec{\Phi}_{E,i}^{\prime }-\vec{\Phi}_{E,i}\in
V_{H}\cap V_{H}^{0}=\{0\}$. Thus,  $\vec{\Phi}_{E,i}$ coincides with $\vec{%
\Phi}_{E,i}^{\prime }$ whose support is in $\omega _{E}$ by construction.
\end{proof}

\begin{remark}\label{FB}
The explicit construction of the basis functions introduced above is as
follows: for any $E\in \mathcal{E}_{H}$ we construct $\vec{\Phi}%
_{E,i}:\Omega \rightarrow \mathbb{R}^{d}$ and the accompanying pressure $\pi
_{E,i}:\Omega ^{\varepsilon }\rightarrow \mathbb{R}$ such that $\vec{\Phi}%
_{E,i}$ and $\pi _{E,i}$ vanish outside the two triangles $T_{1},T_{2}$
adjacent to $E$ and they solve on each of this two triangles$:$%
\begin{eqnarray*}
-\Delta \vec{\Phi}_{E,i}+\nabla \pi _{E,i} &=&0\text{ on }\Omega
^{\varepsilon }\cap T_{k} \\
\func{div}\vec{\Phi}_{E,i} &=&const\text{ on }\Omega ^{\varepsilon }\cap
T_{k} \\
\vec{\Phi}_{E,i} &=&0\text{ on }B^{\varepsilon }\cap T_{k} \\
n\cdot \nabla \vec{\Phi}_{E,i}-\pi _{E,i}n &=&const\text{ on }F\cap \Omega
^{\varepsilon }\text{ for all }F\in \mathcal{E}(T_{k}) \\
\int_{F}\vec{\Phi}_{E,i} &=&\left\{ 
\begin{array}{c}
\vec{e}_{i},~F=E \\ 
0,~F\not=E%
\end{array}%
\right. \text{ for all }F\in \mathcal{E}(T_{k}) \\
\int_{\Omega ^{\varepsilon }\cap T_{k}}\pi _{E,i} &=&0
\end{eqnarray*}
\textcolor{black}{
In the weak form this gives: find $\vec{\Phi}_{E,i}\in H^{1}(T_{k})$ such
that $\vec{\Phi}_{E,i}=0$ on $T_{k}\cap B^{\varepsilon }$, $\pi _{E,i}\in
\LL_{0}(T_{k}\cap \Omega ^{\varepsilon })$ and the Lagrange multipliers $%
\vec{\lambda}_{F}$, $F\in \mathcal{E}(T_{k})$, satisfying%
\begin{eqnarray*}
\int_{T_k\cap \Omega ^{\varepsilon }}\nabla \vec{\Phi}_{E,i} :\nabla \vec{v}%
-\int_{T_k\cap \Omega ^{\varepsilon }}\pi _{E,i}\mathop{\rm div}\nolimits\vec{v%
}+\sum_{F\in \mathcal{E}(T_{k})}\vec{\lambda}_{F}\cdot \int_{F}\vec{\Phi}_{E,i}&=&0, \\
\quad \forall \vec{v}\in H^{1}(T_{k})~\text{such that }\vec{v}&=&0~\text{on }T_{k}\cap B^{\varepsilon } \\
\int_{T_k\cap \Omega ^{\varepsilon }}q\mathop{\rm div}\nolimits\vec{\Phi}%
_{E,i} &=&0,\quad \forall q\in \LL_{0}(T_{k}\cap \Omega ^{\varepsilon }) \\
\sum_{F\in \mathcal{E}(T_{k})}\vec{\mu}_{F}\cdot \int_{F}\vec{\Phi}_{E,i} &=&%
\vec{\mu}_{E}\cdot \vec{e}_{i},~\forall \vec{\mu}_{F}\in \mathbb{R}^{d},~F\in 
\mathcal{E}(T_{k})
\end{eqnarray*}%
We remind that $\LL_{0}(T_{k}\cap \Omega ^{\varepsilon })=\{q\in
\LL (T_{k}\cap \Omega ^{\varepsilon }):\int_{\Omega ^{\varepsilon }\cap
T_{k}}q=0\}.$
}

This gives also an explicit formula for the operator $\pi _{H}:$%
\begin{equation*}
\pi _{H}\left( \sum_{E,i}u_{E,i}\vec{\Phi}_{E,i}\right)
=\sum_{E,i}u_{E,i}\pi _{E,i}
\end{equation*}
\end{remark}

\begin{remark}
The space $V_{H}$ is reduced to the classical Crouzeix-Raviart finite
element space in the case without holes $B^{\varepsilon }=\varnothing $.
Indeed, it is easy to see that the basis functions constructed above can be
written in this case as $\vec{\Phi}_{E,i}=\Phi _{E}\vec{e}_{i}$ where $\Phi
_{E}$ is linear on any triangle $T\in \mathcal{T}_{H}$, discontinuous across
the edges, and such that $\int_{E}\Phi _{E}=1$ and $\int_{E^{\prime }}\Phi
_{E}=0$ for all $E^{\prime }\in \mathcal{E}_{H}$, $E^{\prime }\not=E$. 
\end{remark}

\subsection{Crouzeix-Raviart MsFEM coarse solution}

We now define the the MsFEM solution to problem (\ref{main})--(\ref{mainBC}) as $(\vec{u}%
_{H},p_{H})\in X_{H}$ such that  
\begin{equation}\label{maindiscret}
c((\vec{u}_{H},p_{H}),(\vec{v}_{H},q_{H}))=(f,\vec{v}_{h}),\quad \forall (%
\vec{v}_{H},q_{H})\in X_{H}  
\end{equation}%
We remind that $\vec{u}_{H},p_{H}$ can be represented as $\vec{u}_{H}\in
V_{H}$ and $p_{H}=\pi _{H}(\vec{u}_{H})+\bar{p}_{H}$ with $\bar{p}_{H}\in
M_{H}$. We have also $(\pi _{H}(\vec{u}_{H}),\func{div}\vec{v}_{H})=0$ for
all $\vec{u}_{H},\vec{v}_{H}\in V_{H}$. Hence the problem above can be
recast as: find $\vec{u}_{H}\in V_{H}$ and $\bar{p}_{H}\in M_{H}$ such that%
\begin{eqnarray*}
\int_{\Omega ^{\varepsilon }}\nabla \vec{u}_{H} : \nabla \vec{v}%
_{H}-\int_{\Omega ^{\varepsilon }}\bar{p}_{H}\func{div}\vec{v}%
_{H} &=& \int_{\Omega ^{\varepsilon }}\vec{f}\cdot \vec{v}_{H},\quad \forall 
\vec{v}_{H}\in V_{H} \\
\int_{\Omega ^{\varepsilon }}\bar{q}_{H}\func{div}\vec{u}_{H} &=&0,\quad
\forall \bar{q}_{H}\in M_{H}
\end{eqnarray*}

We remark that the second equation in the system above entails $%
\func{div}\vec{u}_{H}=0$ on $\Omega ^{\varepsilon }$ since $\func{div}%
V_{H}=M_{H}$. We can thus eliminate the pressure from this system by
introducing the subspace of $V_{H}$ consisting of divergence free functions:
denote%
\begin{equation*}
Z_{H}=\{\vec{v}_{H}\in V_{H\text{ }}\text{such that }\func{div}\vec{v}%
_{H}=0\}.
\end{equation*}%
The discrete velocity solution can then be alternatively  defined as: $\vec{u%
}_{H}\in Z_{H}$ such that%
\begin{equation*}
\int_{\Omega ^{\varepsilon }}\nabla \vec{u}_{H}:\nabla \vec{v}%
_{H}=\int_{\Omega ^{\varepsilon }}\vec{f}\cdot \vec{v}_{H},\quad \forall 
\vec{v}_{H}\in Z_{H}.
\end{equation*}%
This ensures immediately existence and uniqueness of $\vec{u}_{H}$.
Existence and uniqueness of $p_{H}$ now follows from the fact $\func{div}%
V_{H}=M_{H}$. 

\subsection{The inf-sup properties for the spaces $X_{H}$ and $X_{H}^{0}$}

We assume from now on that for any triangle $T\in \mathcal{T}_{H}$

\begin{equation}
\inf_{p\in \LL_{0}(T)}\sup_{\vec{v}\in V_{\int 0}(T)}\frac{\int_{T\cap
\Omega ^{\varepsilon }}p\func{div}\vec{v}}{\Vert p\Vert _{\LL (T\cap \Omega
^{\varepsilon })}|\vec{v}|_{H^{1}(T)}}\geq \gamma _{T}>0.  \label{infsupelem}
\end{equation}%
with a mesh-independent and uniform over triangles constant $\gamma _{T}$. Following \cite{ern}, it is
easy to check that this implies%
\begin{equation}
\inf_{(\vec{u},p)\in X_{H}}\sup_{(\vec{v},q)\in X_{H}}\frac{c((\vec{u},p),(%
\vec{v},q))}{\Vert \vec{u},p\Vert _{M}\Vert \vec{v},q\Vert _{X}}\geq \gamma
>0  \label{infsupH}
\end{equation}%
and 
\begin{equation}
\inf_{(\vec{u},p)\in X_{H}^{0}}\sup_{(\vec{v},q)\in X_{H}^{0}}\frac{c((\vec{u%
},p),(\vec{v},q))}{\Vert \vec{u},p\Vert _{M}\Vert \vec{v},q\Vert _{X}}\geq
\gamma >0  \label{infsupH0}
\end{equation}%
with $\gamma $ depending only on $\gamma _{T}$.

\subsection{Some error estimates}

\quad

\begin{lemma}
Let $(\vec{u},p)\in X$ be the sufficiently smooth exact solution to (\ref%
{main})-(\ref{mainBC}) and $(\vec{u}_{H},p_{H})\in X_{H}$ be the discrete solution defined
by (\ref{maindiscret}). Then, assuming (\ref{infsupelem}), 
\begin{equation*}
|\vec{u}-\vec{u}_{H}|_{H^1}+||p-p_{H}||_{\LL} 
\leq CH||\vec{f}||_{\LL }+CH^{%
\frac{1}{2}}\inf_{\vec{c}_{E}\in \mathbb{R}^{d},~E\in \mathcal{E}_{H}}\left(
\sum_{E\in \mathcal{E}_{H}}||\nabla \vec{u}\,n-pn-\vec{c}%
_{E}||_{\LL (E)}^{d}\right) ^{\frac{1}{2}}
\end{equation*}%
with a constant $C>0$ that depend only on the inf-sup constant $\gamma_T $ in (%
\ref{infsupelem}) and on the regularity of the mesh.
\end{lemma}

\begin{proof}
As usual, to prove an a priori error estimate we are going first to
construct an interpolant of the exact solution $(\vec{u},p)\in X$ and then
invoke a (slightly generalized) version of Cea lemma. A good interpolant $(%
\vec{u}_{I},p_{I})\in X_{H}$ may be introduced by requiring 
\begin{equation*}
(\vec{u},p)=(\vec{u}_{I},p_{I})+(\vec{u}^{0},p^{0})
\end{equation*}%
with $(\vec{u}^{0},p^{0})\in X_{H}^{0}$. Such a unique decomposition of $(%
\vec{u},p)$ exists and is unique due to the decomposition of the space $%
X_{H}^{ext}$, cf. (\ref{XHdirect}). It is easy to see that $\func{div}\vec{u}%
^{0}=0$. Indeed, for all $q_{0}\in M_{H}^{0}$%
\begin{equation*}
c((\vec{u}^{0},p^{0}),(0,q^{0}))=c((\vec{u},p),(0,q^{0}))
\end{equation*}%
i.e.%
\begin{equation*}
\sum_{T\in \mathcal{T}_{H}}\int_{T\cap \Omega ^{\varepsilon }}q^{0}\func{div}%
\vec{u}^{0}=0
\end{equation*}%
since $\func{div}\vec{u}=0$. This implies $\func{div}\vec{u}^{0}=c_{T}=const$
on $T\cap \Omega ^{\varepsilon }$ for all $T\in \mathcal{T}_{H}$. But, $%
\int_{\partial T}\vec{u}^{0}\cdot n=0$ due to the definition of $X_{H}^{0}$
so that $c_{T}=0$.

Owing to the inf-sup property (\ref{infsupH0}), we can find $(\vec{v}%
^{0},q^{0})\in X_{H}^{0}$ with $||\vec{v}^{0},q^{0}||_{X}=1$ such that%
\begin{eqnarray*}
\gamma ||\vec{u}-\vec{u}_{I},p-p_{I}||_{X} &=&\gamma ||\vec{u}%
^{0},p^{0}||_{X}\leq c((\vec{u}^{0},p^{0}),(\vec{v}^{0},q^{0})) \\
&=&\sum_{T\in \mathcal{T}_{H}}\int_{T\cap \Omega ^{\varepsilon }}(-\Delta 
\vec{u}^{0}+\nabla p^{0})\cdot \vec{v}^{0}+\sum_{T\in \mathcal{T}%
_{H}}\int_{\partial T\cap \Omega ^{\varepsilon }}(\nabla \vec{u}%
^{0}n-p^{0}n)\cdot \vec{v}^{0} \\
&=&\sum_{T\in \mathcal{T}_{H}}\int_{T\cap \Omega ^{\varepsilon }}\vec{f}%
\cdot \vec{v}^{0}+\sum_{T\in \mathcal{T}_{H}}\sum_{E\in \mathcal{E}%
(T)}\int_{E\cap \Omega ^{\varepsilon }}(\nabla \vec{u}\,n-pn-\vec{c}%
_{E})\cdot \vec{v}^{0}
\end{eqnarray*}%
since on each triangle $T\in \mathcal{T}_{H}$ \ 
\begin{equation*}
-\Delta \vec{u}^{0}+\nabla p^{0}=(-\Delta \vec{u}+\nabla p)-(-\Delta \vec{u}%
_{I}+\nabla p_{I})=\vec{f}
\end{equation*}%
and, moreover, on each edge $E\in \mathcal{E}_{H}$%
\begin{eqnarray*}
\int_{E\cap \Omega ^{\varepsilon }}(\nabla \vec{u}^{0}\,n-p^{0}n)\cdot \vec{v%
}^{0} &=&\int_{E\cap \Omega ^{\varepsilon }}(\nabla \vec{u}\,n-pn)\cdot \vec{%
v}^{0}-\int_{E\cap \Omega ^{\varepsilon }}(\nabla \vec{u}_{I}\,n-p_{I}n)%
\cdot \vec{v}^{0} \\
&=&\int_{E\cap \Omega ^{\varepsilon }}(\nabla \vec{u}\,n-pn)\cdot \vec{v}%
^{0}-\vec{\lambda}_{E}\cdot \int_{E\cap \Omega ^{\varepsilon }}\vec{v}^{0} \\
&=&\int_{E\cap \Omega ^{\varepsilon }}(\nabla \vec{u}\,n-pn-\vec{c}%
_{E})\cdot \vec{v}^{0}
\end{eqnarray*}%
with some constant $\vec{\lambda}_{E}$ depending on $(\vec{u}_{I},p_{I})\in
X_{H}$ (cf. the definition of $X_{H}$) \ and an arbitrary constant $\vec{c}%
_{E}$ (cf. the definition of $X_{H}^{0}$). We now remind the Poincar\'{e}
type inequality%
\begin{equation}
||\vec{v}^{0}||_{\LL (T)}\leq CH|\vec{v}^{0}|_{H^{1}(T)}  \label{poin}
\end{equation}%
and the trace inequality%
\begin{equation}
||\vec{v}^{0}||_{\LL (\partial T)}\leq CH^{\frac{1}{2}}|\vec{v}%
^{0}|_{H^{1}(T)}  \label{trace}
\end{equation}%
which are valid since the average $\vec{v}^{0}$ on the edges composing $%
\partial T$ is zero, see (\cite{MsFEMCR1}) for more details. This allows us to conclude%
\begin{eqnarray*}
\gamma ||\vec{u}-\vec{u}_{I},p-p_{I}||_{X} &\leq &\sum_{T\in \mathcal{T}%
_{H}}||\vec{f}||_{\LL (T)}||\vec{v}^{0}||_{\LL (T)}
\\&+&\sum_{T\in \mathcal{T}_{H}}\left( \sum_{E\in \mathcal{E}(T)}||\nabla \vec{u}\,n-pn-\vec{c}%
_{E}||_{\LL (E)}^{2}\right) ^{\frac{1}{2}}||\vec{v}^{0}||_{\LL (\partial
T)} \\
&\leq &CH||\vec{f}||_{\LL (\Omega )}|\vec{v}^{0}|_{H^{1}(\Omega )}\\
&+&CH^{\frac{1}{2}}\left( \sum_{E\in \mathcal{E}_{H}}||\nabla \vec{u}\,n-pn-\vec{c}%
_{E}||_{\LL (E)}^{2}\right) ^{\frac{1}{2}}|\vec{v}^{0}|_{H^{1}(\Omega )}
\end{eqnarray*}%
Finally,%
\begin{equation}
||\vec{u}-\vec{u}_{I},p-p_{I}||_{X}\leq CH||\vec{f}||_{\LL (\Omega )}+CH^{%
\frac{1}{2}}\left( \sum_{E\in \mathcal{E}_{H}}||\nabla \vec{u}\,n-pn-\vec{c}%
_{E}||_{\LL (E)}^{2}\right) ^{\frac{1}{2}}  \label{estinterp}
\end{equation}%
since $|\vec{v}^{0}|_{H^{1}(\Omega )}\leq ||\vec{v}^{0},q^{0}||_{X}=1$.

We use now the $c$-orthogonality between  $(\vec{u}_{I},p_{I})$ and $(\vec{u}%
^{0},p^{0})\ $\ to write for any $\forall (\vec{v}_{H},q_{H})\in X_{H}$%
\begin{equation*}
c((\vec{u}_{I},p_{I}),(\vec{v}_{H},q_{H}))=c((\vec{u},p),(\vec{v}%
_{H},q_{H}))=\int_{\Omega ^{\varepsilon }}\vec{f}\cdot \vec{v}%
_{H}+\sum_{T\in \mathcal{T}_{H}}\int_{\partial T\cap \Omega ^{\varepsilon
}}(\nabla \vec{u}\,n-pn)\cdot \vec{v}_{H}
\end{equation*}%
Comparing it with (\ref{maindiscret}) we conclude%
\begin{equation*}
c((\vec{u}_{I}-\vec{u}_{H},p_{I}-p_{H}),(\vec{v}_{H},q_{H}))=\sum_{T\in 
\mathcal{T}_{H}}\int_{\partial T\cap \Omega ^{\varepsilon }}(\nabla \vec{u}%
\,n-pn)\cdot \vec{v}_{H},\quad \forall (\vec{v}_{H},q_{H})\in X_{H}.
\end{equation*}%
Let for any $\vec{v}\in X_{H}^{ext}$ denote by $\{\vec{v}\}_{E}$ the average
of $\vec{v}$ over an edge $E$. We can then rewrite the sum of boundary
integrals as%
\begin{equation*}
\sum_{T\in \mathcal{T}_{H}}\int_{\partial T\cap \Omega ^{\varepsilon
}}(\nabla \vec{u}\,n-pn)\cdot \vec{v}_{H}=\sum_{T\in \mathcal{T}%
_{H}}\sum_{E\in \mathcal{E}(T)}\int_{E\cap \Omega ^{\varepsilon }}(\nabla 
\vec{u}\,n-pn)\cdot (\vec{v}_{H}-\{\vec{v}_{H}\}_{E})
\end{equation*}%
Indeed, any edge $E$ comes into these sums two times as contributions from
the two adjacent triangles and thus the additions with $\{\vec{v}_{H}\}_{E}$
cancel each other.   We can now proceed by inserting an arbitrary constant $%
\vec{c}_{E}$ on each edge:%
\begin{equation*}
\sum_{T\in \mathcal{T}_{H}}\int_{\partial T\cap \Omega ^{\varepsilon
}}(\nabla \vec{u}\,n-pn)\cdot \vec{v}_{H}=\sum_{T\in \mathcal{T}%
_{H}}\sum_{E\in \mathcal{E}(T)}\int_{E\cap \Omega ^{\varepsilon }}(\nabla 
\vec{u}\,n-pn-\vec{c}_{E})\cdot (\vec{v}_{H}-\{\vec{v}_{H}\}_{E})
\end{equation*}%
Finally, reusing Poincar\'{e} type inequality (\ref{poin}) and the trace
inequality (\ref{trace}) we conclude in a similar way as above that%
\begin{eqnarray*}
c((\vec{u}_{I}-\vec{u}_{H},p_{I}-p_{H}),(\vec{v}_{H},q_{H}))\leq CH^{\frac{1%
}{2}}\left( \sum_{E\in \mathcal{E}_{H}}||\nabla \vec{u}\,n-pn-\vec{c}%
_{E}||_{\LL (E)}^{2}\right) ^{\frac{1}{2}}|\vec{v}_{H}|_{H^{1}(\Omega
)},\\
\quad \forall (\vec{v}_{H},q_{H})\in X_{H}.
\end{eqnarray*}%
This entails by the inf-sup propoerty (\ref{infsupH}) 
\begin{equation}
||\vec{u}_{I}-\vec{u}_{H},p_{I}-p_{H}||_{X}\leq CH^{\frac{1}{2}}\left(
\sum_{E\in \mathcal{E}_{H}}||\nabla \vec{u}\,n-pn-\vec{c}%
_{E}||_{\LL (E)}^{2}\right) ^{\frac{1}{2}}  \label{estcea}
\end{equation}%
Combining the estimates (\ref{estinterp}) and (\ref{estcea}) with the
triangle inequality yields%
\begin{equation*}
||\vec{u}-\vec{u}_{H},p-p_{H}||_{X}\leq CH||\vec{f}||_{\LL (\Omega )}+CH^{%
\frac{1}{2}}\left( \sum_{E\in \mathcal{E}_{H}}||\nabla \vec{u}\,n-pn-\vec{c}%
_{E}||_{\LL (E)}^{2}\right) ^{\frac{1}{2}}
\end{equation*}
\end{proof}

The error estimate of the preceding Lemma is not completely satisfactory because of the oscillating nature of the exact solution which can make the norms on the edges quite big. Unfortunately, this phenomenon can be fully investigated theoretically only in the simplest setting of periodically placed holes. We announce here such a result which will be proved in the subsequent paper \cite{part2} following the ideas of \cite{MsFEMCR2}.
\begin{lemma}  
Let $B^{\epsilon}$ be the set of holes placed periodically in both directions at the distance $\epsilon$ of each other.  Assume (\ref{infsupelem}) and some technical hypotheses about the mesh as in \cite{MsFEMCR2}. Then the MsFEM solution satisfies
\begin{equation*}
|\vec{u}-\vec{u}_{H}|_{H^1}+||p-p_{H}||_{\LL}
\leq 
C\left(H +\sqrt\epsilon +\sqrt{\frac \epsilon H} \right) ||\vec{f}||_{H^{2}}
\end{equation*}
\end{lemma}  
}

\begin{figure}[hbt]
\centering
\includegraphics[width = 5in]{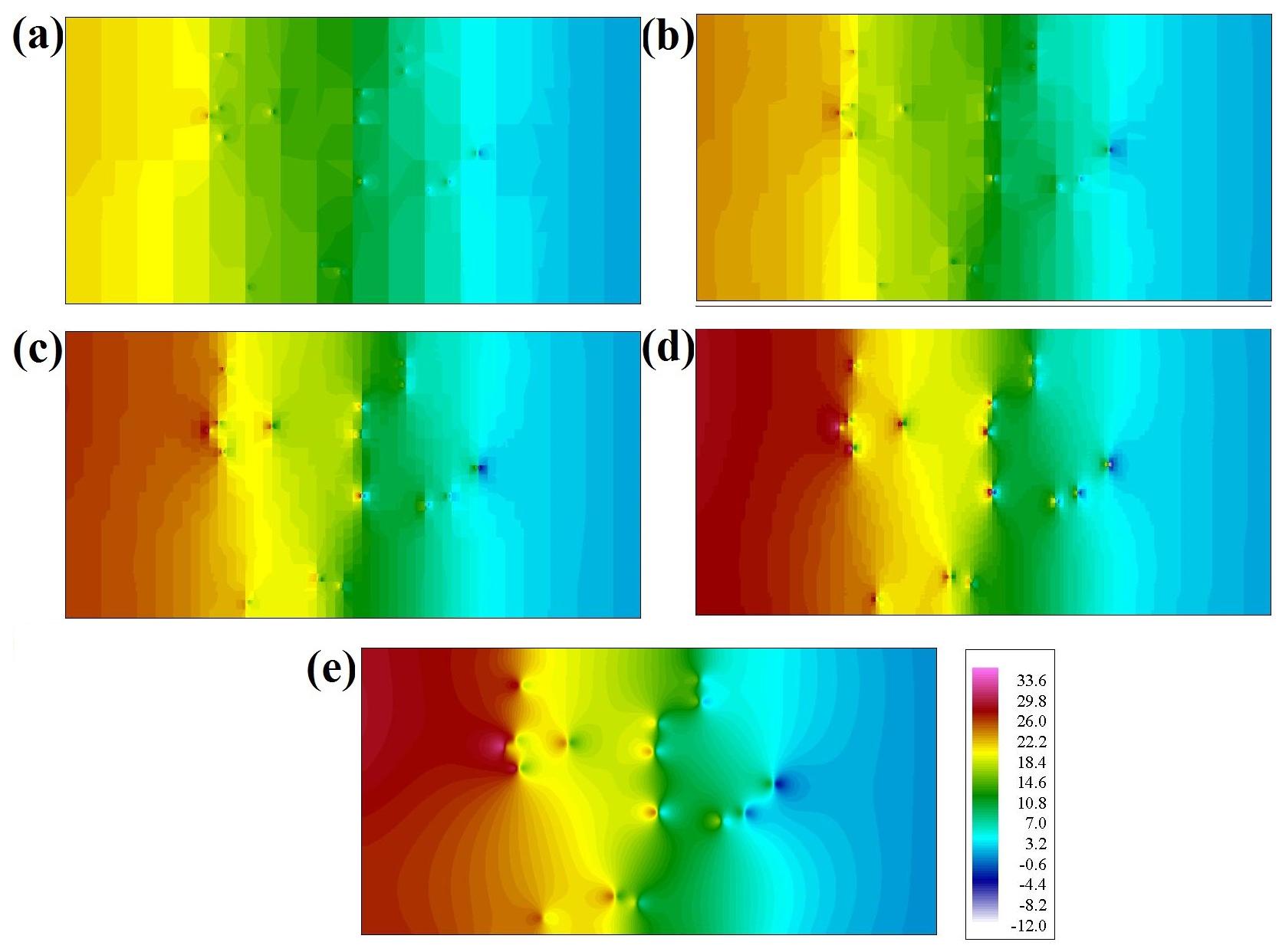} 
\caption{Pressure contours. Calculations done on (a) $8\times 16$, (b) $16\times 32$, (c) $32\times 64$, (d)$64\times128$, compared with (e) reference pressure computed on $1280 \times 640$}
\label{presrecon}
\end{figure}
\begin{figure}[hbt]
\centering
\includegraphics[width = 5in]{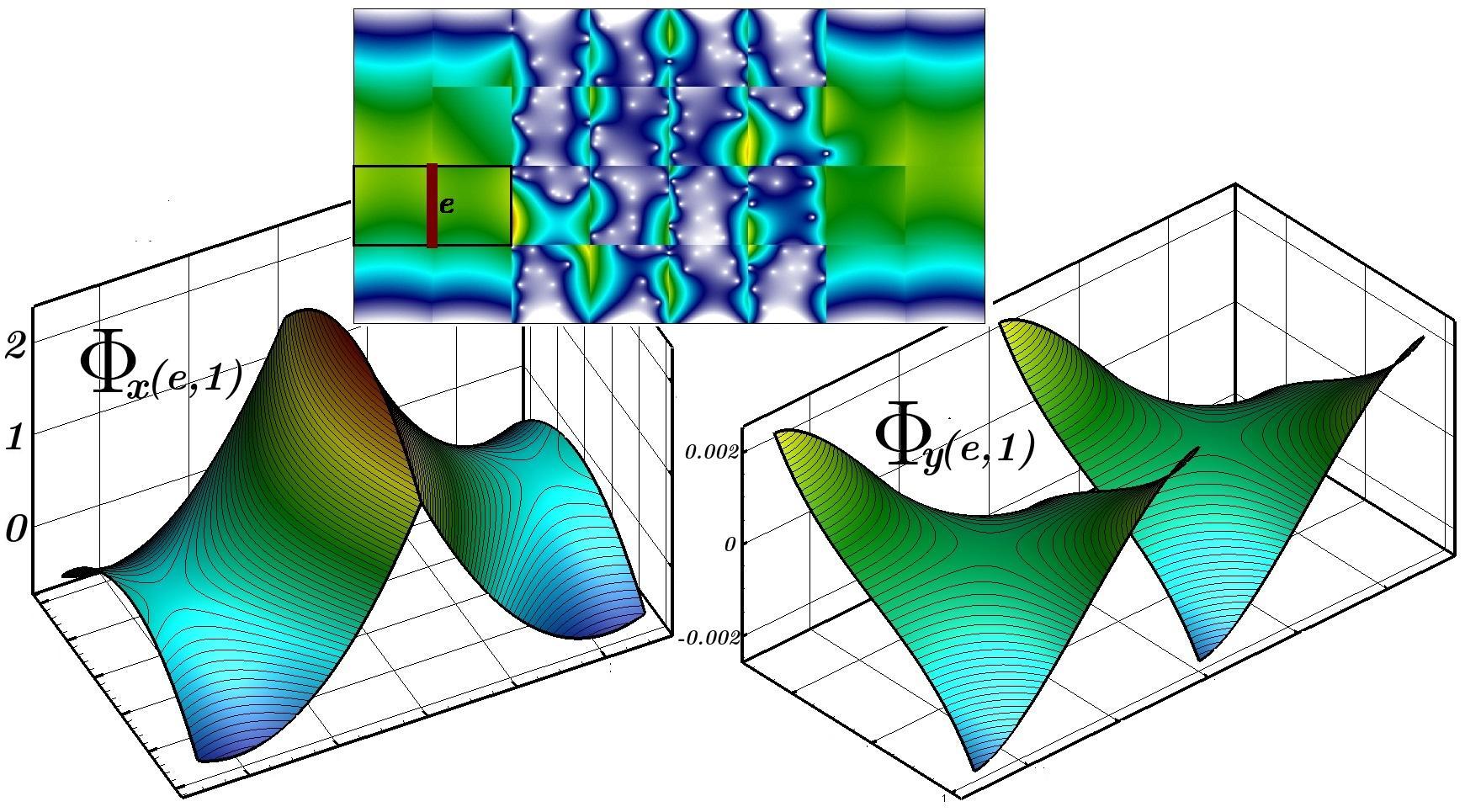} 
\caption{Crouzeix-Raviart MsFEM basis function for elements without obstacles}
\label{qnonperfor}
\end{figure} 
\begin{figure}[hbt]
\centering
\includegraphics[width = 5in]{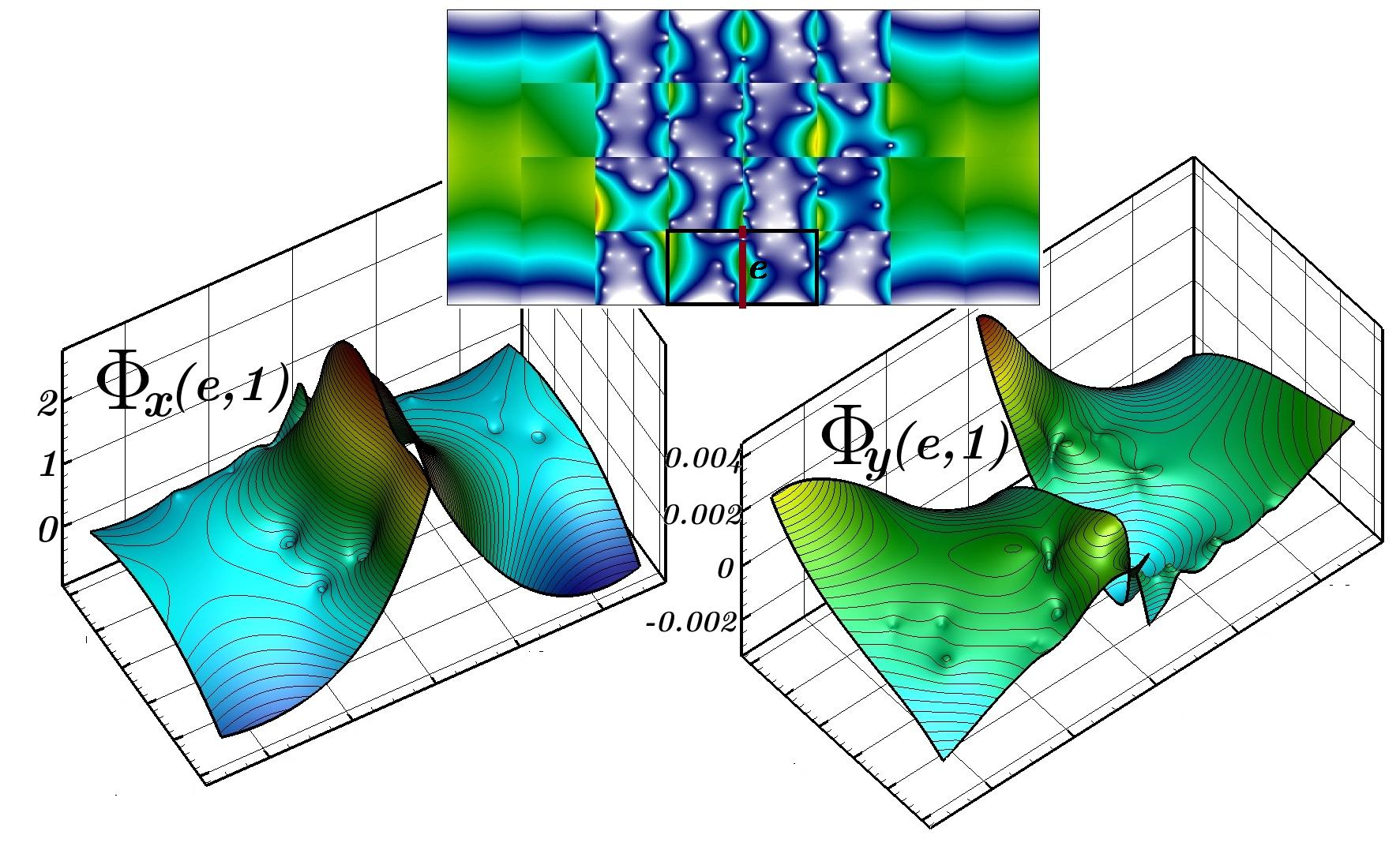} 
\caption{Crouzeix-Raviart MsFEM basis function for elements with arbitrarily placed obstacles}
\label{qperfor}
\end{figure} 

\section{Implementation issues}\label{implement} 
The Crouzeix-Raviart MsFEM as presented so far is not directly implementable since it relies on the exact solutions of the local problems in the construction of the basis functions, cf. Remark \ref{FB}. In practice, these problems should be discretized on mesh sufficiently fine to resolve the geometry of obstacles. For highly non-periodic pattern of obstacles, complicated body-fitted unstructured mesh is likely what engineers would resort to. In this article, we opt for performing our computations on a simple uniform Cartesian grid using the penalization method. We shall explain first how to do it when solving the original problem (\ref{main})--(\ref{mainBC}) as a whole without resorting to the MsFEM technique (this is needed any way to construct the reference solutions in our numerical experiments). 
\subsection{Application of penalization method to compute the reference solution}
\label{sec:applicationpenal}
To avoid complex and ad-hoc grid generation methods when solving (\ref{main})--(\ref{mainBC}) in $\Omega^\epsilon$ we replace it with the following penalized problem
\begin{eqnarray}
-\nabla\cdot(\nu^\kappa \nabla \vec{u}) + \sigma^\kappa\vec{u} + \nabla p = \vec{f}^\kappa & \hspace{5mm}\textrm{in}\hspace{5mm} & \Omega  \label{penal}\\
\nabla\cdot\vec{u} = 0 &\hspace{5mm}\textrm{in}\hspace{5mm}&\Omega  \nonumber\\
\vec{u}=\vec{w} & \hspace{5mm}\textrm{on}\hspace{5mm} & \partial\Omega \nonumber
\end{eqnarray}
in which
\begin{eqnarray}
\nu^\kappa = \left\{ \begin{array}{rl}
 \frac{1}{h} &\mbox{ in $B^\epsilon$} \\
 \nu &\mbox{ in $\Omega^\epsilon$}
       \end{array}  \right.,
\quad
\sigma^\kappa = \left\{ \begin{array}{rl}
 \frac{1}{h^3} &\mbox{ in $B^\epsilon$} \\
 0 &\mbox{ in $\Omega^\epsilon$}
       \end{array}  \right.,
\quad
f^\kappa = \left\{ \begin{array}{rl}
 0 &\mbox{ in $B^\epsilon$} \\
 f &\mbox{ in $\Omega^\epsilon$}
       \end{array}  \right..
\end{eqnarray}
Here $h$ is the size of the fine scale mesh element used to capture highly oscillatory solution. The penalization coefficient $\sigma^\kappa$ forces the solution to vanish rapidly inside the obstacles. Other variants of penalization methods are studied in \cite{Bruneau}. 

To calculate the reference solutions, we utilize the Q1-Q1 FEM which has velocity and pressure degrees of freedom defined at the same set of grid points, namely the uniform Cartesian grid $\mathcal{T}_h$ of step $h$. As is well known, cf. \cite{brezzifortin}, this choice of velocity and pressure spaces requires some stabilization which weakens the condition $\nabla \cdot \vec{u} = 0$. The simplest way to achieve this is by perturbing the incompressibility constraint with a pressure Laplacian term, see \cite{brezzipitkaranta}. Our numerical problem reduces thus to finding $\vec{u}_h$ and $p_h$ from the Q1-Q1 finite element space  such that
\begin{align*}
\int_\Omega \nu^\kappa \nabla \vec{u}_h:\nabla\vec{v}_h+\int_\Omega \sigma^\kappa\vec{u}_h\cdot\vec{v}_h-\int_\Omega p_h\nabla\cdot\vec{v}_h &= \int_\Omega \vec{v}_h \cdot \vec{f}, &\quad\forall\vec{v}_h\\
-\int_\Omega q_h\nabla\cdot\vec{u}_h-\theta h^2 \int_\Omega\nabla p_h \cdot\nabla q_h &= 0, &\quad\forall q_h
\end{align*}
where $\theta>0$ is the stabilization parameter and $\vec{v}_h$, $q_h$ go over the same finite element space. Certainly other stable or stabilized elements can be used.

\subsection{Calculation of Crouzeix-Raviart MsFEM basis functions}
\label{sec:reference}

The calculations of the multi-scale basis in this paper are carried out similarly to that of the reference solution. We introduce the fine meshes $\mathcal{T}_h(T)$ for each element $T$ of the coarse mesh $\mathcal{T}_H$ (in all our numerical experiments these local meshes will be just sub-meshes of the global fine mesh $\mathcal{T}_h$). The MsFEM basis functions are then calculated as
follows: for any $E\in \mathcal{E}_{H}$ we construct $\vec{\Phi}_{E,i}$ and the accompanying pressure $\pi_{E,i}$ supported in the two triangles $T_{1},T_{2}$ adjacent to $E$ and belonging to the Q1-Q1 FEM spaces on the meshes $\mathcal{T}_h(T_1)$ and $\mathcal{T}_h(T_2)$. They solve on each of these two triangles:
\begin{align*}
\int_{T_k} \nu^\kappa \nabla \vec{\Phi}_{E,i} :\nabla \vec{v}_h 
+\int_{T_k} \sigma^\kappa \vec{\Phi}_{E,i} \cdot \vec{v}_h 
-\int_{T_k}\pi _{E,i}\nabla\cdot\vec{v}_h +\sum_{F\in \mathcal{E}(T_{k})}\vec{\lambda}_{F}\cdot \int_{F}\vec{\Phi}_{E,i} &=0, 
  \quad\forall \vec{v}_h  \\
-\int_\Omega q_h\nabla\cdot\vec{\Phi}_{E,i}-\theta h^2 \int_\Omega\nabla \pi_{E,i} \cdot\nabla q_h &= 0, \quad\forall q_h\\
\sum_{F\in \mathcal{E}(T_{k})}\vec{\mu}_{F}\cdot \int_{F}\vec{\Phi}_{E,i} &=\vec{\mu}_{E}\cdot \vec{e}_{i} \\
\forall \vec{\mu}_{F}\in \mathbb{R}^{d},&~F\in\mathcal{E}(T_{k})
\end{align*}
where $\vec{v}_h$ and $q_h$ go over the same finite element spaces as $\vec{\Phi}_{E,i}$ and $\pi_{E,i}$. We also recall that $\mathcal{E}(T_{k})$ denotes the set of all the edges of the triangle $T_k$ and $\theta>0$ is the stabilization parameter. 

\subsection{Note on Boundary Condition}\label{nonhomog}
Special treatment on boundary condition is applied i.e., we approximate the strong form of non-homogeneous Dirichlet boundary condition with

\begin{equation}
\label{specialbc}
\int_{e} \vec{u}_H = \int_{e} \vec{w}, \hspace{5mm}\textrm{for}\hspace{1mm}\textrm{all}\hspace{1mm} e \in \mathcal{E}_H \hspace{1mm}\textrm{on}\hspace{1mm}\partial\Omega.
\end{equation}
Equation (\ref{specialbc}) is therefore equivalent with 
\begin{equation}
\vec{u}_H|_{e} = \frac{1}{|e|}\int_{e} {\vec{w}}
\end{equation} 
on each boundary edge $e$.
This approach is a modification with respect to the earlier works in \cite{MsFEMCR1, MsFEMCR2} where the boundary condition were strongly incorporated in the definition of $V_H$.  Our approach therefore gives more flexibility when implementing non zero $g$.  It will be demonstrated in the later sections how the application of this approach on our MsFEM gives conveniently converging results toward the correct solution.

\begin{table}[Hbt!]
\caption{Convergence study of cavity flow}
\begin{center}
\begin{tabular}{{l}{c}{c}{c}{c}{c}}
\hline\hline
Config.      &   Ratio $(H/\epsilon)$  & $L^1$ Rel. & $L^2$ Rel. & $H^1$ Rel. & $L^2$ P (Rel.) \\
\hline
$2\times 4$    & 17.544 & 0.756 & 0.640 & 0.837 & 0.992\\
$4\times 8$    & 8.772  & 0.576 & 0.516 & 0.780 & 0.628\\
$8\times 16$   & 4.386  & 0.477 & 0.396 & 0.625 & 0.480\\
$16\times 32$  & 2.193  & 0.337 & 0.269 & 0.617 & 0.390\\
$32\times 64$  & 1.096  & 0.257 & 0.194 & 0.544 & 0.312\\
$64\times 128$ & 0.548  & 0.160 & 0.102 & 0.493 & 0.288\\
\hline
\end{tabular}
\end{center}
\label{cavityerror}
\end{table}
\begin{table}[Hbt]
\caption{Convergence study of channel flow (case A)}
\begin{center}
\begin{tabular}{{l}{c}{c}{c}{c}{c}}
\hline\hline
Config.      &  Ratio $(H/\epsilon)$  & $L^1$ Rel. & $L^2$ Rel. & $H^1$ Rel. & $L^2$ P (Rel.)\\
\hline
$2\times 4$    & 25.00  & 0.305 & 0.395 &  0.631 & 0.874\\
$4\times 8$    & 12.50  & 0.169 & 0.212 &  0.605 & 0.601\\
$8\times 16$   & 6.250  & 0.110 & 0.142 &  0.594 & 0.563\\
$16\times 32$  & 3.125  & 0.090 & 0.115 &  0.506 & 0.420\\
$32\times 64$  & 1.563  & 0.067 & 0.087 &  0.411 & 0.275\\
$64\times 128$ & 0.781  & 0.043 & 0.062 &  0.320 & 0.141\\
\hline
\end{tabular}
\end{center}
\label{channelerrora}
\end{table}
\begin{table}[Hbt!]
\caption{Convergence study of channel flow (case B)}
\begin{center}
\begin{tabular}{{l}{c}{c}{c}{c}{c}}
\hline\hline
Config.      &  Ratio $(H/\epsilon)$  & $L^1$ Rel. & $L^2$ Rel. & $H^1$ Rel. & $L^2$ P (Rel.)\\
\hline
$2\times 4$    & 120.192 & 0.508 & 0.609 & 0.892 & 0.891\\
$4\times 8$    & 60.096 & 0.321 & 0.423 &  0.805 & 0.800\\
$8\times 16$   & 30.048 & 0.171 & 0.237 &  0.694 & 0.730\\
$16\times 32$  & 15.024 & 0.104 & 0.144 &  0.606 & 0.666\\
$32\times 64$  & 7.512  & 0.080 & 0.110 &  0.561 & 0.490\\
$64\times 128$ & 3.756  & 0.062 & 0.081 &  0.452 & 0.259\\
\hline
\end{tabular}
\end{center}
\label{channelerrorb}
\end{table}

\section{Numerical Results}
\label{sec:numerical}
Homogeneous $\nu = 1$ is assumed throughout our tests. The heterogeneities in the problems are represented by sporadic placements of obstacles. However, the application of oscillating $\nu$ is straightforward. In this paper, reconstruction of fine scale pressure field is not emphasized. In any case, the coarse scale pressure field (element-wise constant) can always be recovered.   
\subsection{Enclosed Flows in Heterogeneous Media}
In the first example, we consider the cavity flow problem. In this problem, all the velocities are known at $\partial\Omega$ (enclosed flow) in which the pressure is unique only up to a constant. The velocities vanish everywhere at boundaries except at the top of the domain $\Omega = [-1\leq x \leq 1, 0 \leq y \leq 1]$ where the tangential velocity there is set to be $u_x=1$. This tangential velocity is the only force that drives the flow. 49 small obstacles, each with width of $0.0285$ are randomly laid within the cavity, see Fig. \ref{cavityperfor}. The reference to which our solutions will be compared is obtained using Q1-Q1 FEM on $640 \times 1280$ elements.

In Figs. \ref{cavityux} and \ref{cavityuy} the solutions of Stokes flow inside the cavity in terms of $u_x$ and $u_y$ contours using several mesh configurations are given alongside that of the reference solution. Solutions on $32 \times 64$ and $64 \times 128$ elements are almost identical to the reference solution calculated on $640 \times 1280$ elements, but in some engineering circumstances, $16 \times 32$ elements already provides quantitatively sufficient measures of the flow. 

In Figs. \ref{cavitystream}, the streamlines and the velocity magnitude contours solved using $32 \times 64$ elements are compared with those of the reference. We notice that important recirculation zones at the top half of the domain can be well captured. In this problem, discontinuities are present on both ends of the top lid and it is shown in Table \ref{cavityerror} that our treatment of boundary conditions is sufficient in providing converging solutions. We describe the errors in terms of error norms $L^1$ relative, $L^2$ relative and $H^1$ relative. We note that with our method, coarse-grid converging results irrespective to positioning of obstacles without any oversampling methods is rather expected. 

\subsection{Open-Channel Flows in Heterogeneous Media}
Unlike the example of an enclosed flow given above, in the second examples we consider open-channel flows. The first past of this test (case A) includes $16$ obstacles each with width $\epsilon = 0.02$. The second part (case B) includes $144$ very fine obstacles each with width $\epsilon = 0.00832$, all randomly laid in the center of a 2-D channel $\Omega = [0\leq x \leq 4, -1 \leq y \leq 1]$, see Fig. \ref{channelperfor}(a)and(b). Parabolic inflow  boundary condition, $u_x = 1-y^2; u_y = 0$, is taken at the inlet (left) whereas natural boundary condition $\partial u/\partial n = 0$ is assumed at the outlet (right). At the top and bottom walls, no slip boundary conditions $\vec{u}=0$ are taken. 

In Figs. \ref{channeluxa} and \ref{channeluya} the case A solutions of our method on several mesh configurations in terms of $u_x$ and $u_y$ contours are given alongside those of the reference. As in the cavity flow example, the reference solution is calculated on $640\times 1280$ elements each with width of $h=0.003125$. As shown in the example of cavity flow, our results show converging behaviour toward reference solutions. Most of the flow features can be quantitatively obtained using mere $16 \times 32$ elements. In Table \ref{channelerrora}, we look at the convergence study of this problem. In Figs. \ref{channelstreama}, we can see the comparison of streamline and velocity magnitude of the solution on $32 \times 64$ compared to that of the reference. In Figs. \ref{presrecon}, the pressure reconstruction for several mesh configurations for this case are given and shown to be converging toward the reference solution where the $L_2$ relative error for configuration $128 \times 64$ is $0.141$. 

For case B, the contours of $u_x$ and $u_y$ in comparison to the reference solution, also calculated on $640\times 1280$ elements each with width of $h=0.003125$, are given in Figs. \ref{channelux} and \ref{channeluy}. It is shown that the method is able to recover the presence of very fine obstacles and important flow features already at $16 \times 32$. In Figs. \ref{channelstream}, the streamlines calculated on $32 \times 64$ coarse mesh are compared with those of the reference on top of velocity magnitude contour and are in good agreement. The error study of case B is shown in Table \ref{channelerrorb} showing converging results as exhibited at other cases even at considerably much larger ratios $(H/\epsilon)$.

In Figs. \ref{qnonperfor} and \ref{qperfor} we illustrate the Crouzeix-Raviart MsFEM basis function when obstacles are absent and present. We can see that the weak conformity of Crouzeix-Raviart basis function leads to natural boundary condition at the coarse element edges which adapts well with the arbitrary pattern of obstacles.

\section{Concluding remarks}
\label{sec:concluding}
The Crouzeix-Raviart MsFEM has been developed and tested for solving Stokes flow in genuine heterogeneous media. By genuine we mean on circumstances where analytical representation of the microscopic features of the flow are unavailable. This is illustrated by using very fine and non-periodic placements of obstacles. The method has been tested in the context of obstacles-filled cavity flows and open-channel flows. 

The Crouzeix-Raviart MsFEM basis functions are calculated within each coarse elements using stabilized Q1-Q1 FEM. Non-conforming nature of Crouzeix-Raviart element allows the method to accommodate random pattern of obstacles without having to use oversampling methods.  

Penalization method has been seamlessly incorporated in modeling arbitrary pattern of obstacles thus allowing extensive use of simple Cartesian mesh. Convergence studies of enclosed and open flows are given. Additionally, non-homogeneous boundary conditions are considered in the test cases to demonstrate the robustness of the method. Good quantitative agreement has been reached with the reference solution at relatively coarse mesh configurations.

Although the test cases are given in two spatial dimensions, the extension of this work onto three spatial dimensions is straightforward. The calculations of MsFEM basis functions within a coarse element are done independent of its neighboring elements which makes it suitable for the application of parallel programming.

\section{Acknowledgement}
This work is done under the auspices of ''Fondation Sciences et Technologies pour l'Aeronautique et l'Espace'', in the frame of the project 'AGREMEL' (contract \# RTRA-STAE/2011/AGREMEL/02)

\bibliographystyle{siam}
\bibliography{crmsfem}

\end{document}